\documentclass[11pt]{article}
\usepackage{latexsym,amsfonts,amsmath,amsthm,amssymb, epsfig}
\usepackage{color,xcolor}
\usepackage{hyperref}

\usepackage[a4paper,total={170mm,250mm},vcentering,dvips]{geometry}
\newtheorem{thm}{Theorem}
\newtheorem{lem}[thm]{Lemma}

\newtheorem{OP}{Open Problem}

\theoremstyle{remark}
\newtheorem{rem}{Remark}

\theoremstyle{definition}
\newtheorem{dfn}[thm]{Definition}

\newcommand{\norm}[2]{\left\|\left.{#1}\right|{#2}\right\|}

\newcommand{\R}{\mathbb{R}}

\newcommand{\N}{\mathbb{N}}

\newcommand{\E}{\mathbb E}
\newcommand{\var}{\mathrm{var}\,}
\newcommand{\e}{\mathrm e}

\renewcommand{\P}{\mathbb P}

\begin{document}

\title{Path regularity of Brownian motion and Brownian sheet}

\author{H. Kempka\footnote{Ernst-Abbe-Hochschule Jena, University of Applied Sciences, FB Grundlagenwissenschaften, Carl-Zeiss-Promenade 2, 07745 Jena, Germany.
Email: \href{mailto:henning.kempka@eah-jena.de}{henning.kempka@eah-jena.de}
},
\quad C. Schneider\footnote{Friedrich-Alexander University Erlangen-Nuremberg, Applied Mathematics III, Cauerstr. 11, 91058 Erlangen, Germany. Email: \href{mailto:cornelia.schneider@math.fau.de}{cornelia.schneider@math.fau.de}. 
The work of this author has been supported by Deutsche Forschungsgemeinschaft (DFG), Grant No. SCHN 1509/1-2.}, 
\quad and J. Vybiral\footnote{Department of Mathematics, Faculty of Nuclear Sciences and Physical Engineering,
Czech Technical University, Trojanova 13, 12000 Praha, Czech Republic. Email: \href{mailto:jan.vybiral@fjfi.cvut.cz}{jan.vybiral@fjfi.cvut.cz}.
The work of this author has been supported by the European Regional Development Fund-Project ``Center for Advanced Applied Science''
(No. CZ.02.1.01/0.0/0.0/16\_019/0000778) and in part also by the Danube Region Grant no. 8X2043 of the Czech Ministry of Education, Youth and Sports.}}

\maketitle
\date{}

\begin{abstract}

By the work of P. L\'evy, the sample paths of the Brownian motion are known to satisfy a certain H\"older regularity condition almost surely. This was later improved by Ciesielski, who studied the regularity of these paths in Besov and Besov-Orlicz spaces. We review these results and propose new function spaces of Besov type, strictly smaller than those of Ciesielski and L\'evy, where the sample paths of the Brownian motion lie in almost surely. In the same spirit, we review and extend the work of Kamont, who investigated the same question for the multivariate Brownian sheet and function spaces of dominating mixed smoothness.\\

{\em MSC2020 Math Subject Classifications:} 60G17, 60G15, 46E35, 60G60.
\end{abstract}

\tableofcontents 

\section{Introduction}\label{sec:Into}

Already in 1937, Paul L\'evy showed \cite[Section 52]{Levy}, that the sample paths of the Wiener process $W=(W_t)_{t\ge 0}$ satisfy almost surely the H\"older condition
\begin{equation}\label{eq:Levy_orig}
|W_{t'}-W_{t}|\le c\cdot \sqrt{2|t'-t|\log\left(\frac{1}{|t'-t|}\right)}
\end{equation}
for every $c>1$ and $|t'-t|$ small enough. In general, one can define for a positive function $g$ on $[0,1]$ the H\"older space ${\mathcal C}^g([0,1])$ as the collection of all functions $f$ on $[0,1]$ such that
\begin{equation*}
|f(s)-f(t)|\le c\, g(|s-t|)\quad \text{for all}\quad 0\le s,t\le 1.
\end{equation*}
Then the result of L\'evy shows that the paths of $W$ lie almost surely in the H\"older space ${\mathcal C}^g([0,1])$ with $g(r)=|r\log r|^{1/2}$ for  $r>0$ small. Furthermore, this result is known to be optimal and this space is the smallest one in the scale of H\"older spaces, where the paths of $W$ lie in almost surely \cite{CR81,Levy2}. This shows, in particular, that the $\log$-factor in \eqref{eq:Levy_orig} is necessary and that the smoothness regularity $1/2$, which is $\mathcal{C}^{1/2}([0,1])$, i.e., the space above with $g(r)=|r|^{1/2}$, cannot be achieved in the scale of H\"older spaces.

Later on, Ciesielski proved in \cite{Ciesiel91}, that one can actually obtain smoothness of the order $1/2$ if one gives up slightly on the integrability. Namely, \cite{Ciesiel91} shows that the paths of $W$ lie almost surely in the Besov space $B^{1/2}_{p,\infty}([0,1])$ for $1\le p<\infty$. The excluded endpoint space is again $B^{1/2}_{\infty,\infty}([0,1])=\mathcal{C}^{1/2}([0,1])$. Shortly after, Ciesielski and his co-authors \cite{Ciesiel93,CKR93} refined the analysis of \cite{Ciesiel91} and discovered, that almost all paths of $W$ lie in the Besov-Orlicz space $B^{1/2}_{\Phi_2,\infty}([0,1])$, which combines the technique of Besov spaces together with the Orlicz space generated by the Orlicz function $\Phi_2(t)=\exp(t^2)-1$. This space is (properly) included both in the H\"older space ${\mathcal C}^g([0,1])$ discovered by L\'evy as well as in the Besov spaces $B^{1/2}_{p,\infty}([0,1])$ for $1\le p<\infty.$ As such, $B^{1/2}_{\Phi_2,\infty}([0,1])$ represents nowadays the smallest space, where the sample paths of the Wiener process are known to lie in almost surely. On the other hand, its definition is certainly more involved than the H\"older condition of L\'evy \eqref{eq:Levy_orig}.

The results on the regularity of sample Wiener paths were later complemented, generalized, and applied in several different ways. The optimality of the result of Ciesielski in the scale of Besov spaces was studied in \cite{R93}
and \cite{BO11}, {where the latter reference  studies} the topic in the frame of modulation spaces and Wiener amalgam spaces. {Path regularity of more general processes was investigated in \cite{Herren, Schilling1, Schilling3} and we refer to \cite{Veraar} for results on the torus.
The approach was also generalized to Wiener processes with values in Banach spaces in \cite{HV08}, and applied to regularity properties of stochastic differential equations
\cite{OSK18,OV20}.}

{Let us also mention that many other properties of sample paths of the Wiener process, Brownian sheet, and other random processes were studied extensively. They include different dimensions of the graph set, small ball probabilities, hitting probabilities or the law of iterated logarithm. We refer in this context to \cite{Khos, KhosXiao, MP, TudorXiao, Wang07, Xiao1} and the references given therein.}

The first aim of our work is to present an essentially self-contained proof of the results of L\'evy and Ciesielski, which should be easily accessible to readers familiar with the theory of function spaces. This will be done by first deriving  the decomposition of sample Wiener paths into a series of Faber splines with independent standard Gaussian random coefficients.
 {Afterwards, the proof that almost all paths of the Wiener process lie in a H\"older space of L\'evy or in the Besov or Besov-Orlicz spaces of Ciesielski,  reduces to  rather straightforward concentration inequalities for independent Gaussian variables. For this purpose we collect basic facts  on Gaussian variables (and some other related random variables), that are needed throughout the manuscript, in Section \ref{sec:app}.}

Let us briefly summarize the main steps of this approach. It is essentially based on two very well-known properties of the Faber system. This is a system of shifted and dilated hat functions $v_{j,m}$, cf. \eqref{eq:vjm}, where $j\in\N_0$ and $0\le m\le 2^{j}-1$, which are concentrated on the dyadic intervals $I_{j,m}=[m\cdot 2^{-j},(m+1)2^{-j}].$ The first property of this system is described in detail in Theorem \ref{thm:LevyDecomp}. It states, that if $\{\xi_{j,m}:\ j\in\N_0, \  0\le m\le 2^j-1\}$ are independent standard Gaussian variables, then the series
\[
\sum_{j=0}^\infty \sum_{m=0}^{2^j-1} 2^{-(j+2)/2}\xi_{j,m}v_{j,m}(t)\quad\text{ for } t\in[0,1]
\]
converges almost surely uniformly on $[0,1]$ and its limit coincides with the Wiener process $W=(W_t)_{t\ge 0}$.

The second key property of the Faber system is that it can be used to describe several classical function spaces of Besov-type. To be more precise, a function $f$ representable (in some sense) by the series
\begin{equation*}
f=\sum_{j=0}^\infty \sum_{m=0}^{2^j-1}\mu_{j,m}2^{-js}v_{j,m}
\end{equation*}
belongs to such a Besov-type space if, and only if, the sequence of coefficients $\{\mu_{j,m}\}_{j,m}$ satisfies some summability and/or integrability condition. Naturally, these conditions differ from one space to another, but usually they can be rewritten in the language of the step functions
\begin{equation}\label{eq:Intro01}
f_j=\sum_{m=0}^{2^j-1}\mu_{j,m}\chi_{j,m},
\end{equation}
where $\chi_{j,m}$ is the characteristic function of $I_{j,m}.$ For example, the proof that Wiener paths lie almost surely in the Besov-Orlicz space $B^{1/2}_{\Phi_2,\infty}([0,1])$, reduces by this technique  to the statement, that $\|f_j\|_{\Phi_2}$ is finite and uniformly bounded over $j\in\N_0$ if we replace the $\mu_{j,m}$'s in \eqref{eq:Intro01} by independent standard Gaussian variables $\xi_{j,m}.$ We use this approach to re-prove the results of L\'evy and Ciesielski and, in the named case of the Besov-Orlicz space $B^{1/2}_{\Phi_2,\infty}([0,1])$, we provide an alternative proof, based on a characterization of the Orlicz space $L_{\Phi_2}([0,1])$ in terms of non-increasing rearrangements.

The second aim of this paper is to show that this procedure can be stepped up and that  {(based on some knowledge about  Gaussian variables)} one can produce even smaller function spaces, which still contain the sample Wiener paths almost surely. The price to pay in this context is that the new spaces do not fall into any standard scale of function spaces. Let us again briefly sketch the main idea and the main results.
First, we observe that if we use independent Gaussian variables as the coefficients in \eqref{eq:Intro01}, then the Orlicz space $L_{\Phi_2}([0,1])$ measures very effectively the size of the $f_j$'s among the function spaces invariant with respect to the rearrangement of a function. But it does not take  any effort to describe the position of large values of $f_j$. Indeed, if $\xi_j=(\xi_{j,0},\dots,\xi_{j,2^j-1})$ are independent Gaussian variables, then the maximum of $|\xi_{j,m}|$ over $m$ is known to behave asymptotically like $\sqrt{j}$ with high probability. But these large values are unlikely to appear close to each other in $\xi_j$. Therefore, we expect that the averages of randomly constructed $f_j$ would be of much smaller size than the $f_j$'s themselves. This is indeed the case, as is shown in Theorem \ref{thm:Aepsilon}, where we prove that $\|A_k f_j\|_{\Phi_2}$ behaves (up to a polynomial factor) as $2^{(k-j)/2}$ for $0\le k \le j$. Here, $A_k g$ is the average of a function $g$ over the dyadic intervals $I_{k,l}$, cf. \eqref{eq:AV1}. In Theorem \ref{thm:13} and Theorem \ref{thm:15} we provide  three more function spaces of this kind, including certain function spaces based on  some sort of ball means of differences.

We study also the generalization of the previous approach to the multivariate setting. The high-dimensional analogue of the Wiener process is known as Brownian sheet, cf. Definition \ref{dfn:BrownianSheet}. For the sake of brevity, we restrict ourselves to $d=2$ and the Brownian sheet defined on the unit square $[0,1]^2$ of $\R^2$, but higher dimensions could be treated in the same way with only minor modifications. The known results in this area go essentially back to the work of Anna Kamont \cite{Kamont2, Kamont1} (whose Ph.D. supervisor was Zbigniew Ciesielski).

Similarly to the one-dimensional case, one first obtains a decomposition of the paths of the Brownian sheet in a suitable basis, the so-called multivariate Faber system. This is nothing else than the tensor products of the hat functions of the one-dimensional Faber system. The coefficients in this decomposition are again independent standard Gaussian variables. The corresponding function spaces, the spaces of dominating mixed smoothness, are very well-known in the field of approximation theory of functions of several variables. Unfortunately, in \cite{Kamont2} Kamont called these spaces \emph{anisotropic H\"older classes}, which might explain why her work went essentially unnoticed by the community of researchers investigating function spaces of dominating mixed smoothness.

Also in this part we re-prove the known results in a way which we hope will be easily accessible for readers with a background in the theory of function spaces. Again, we employ a number of different scales of function spaces to describe the path regularity of the Brownian sheet. These include Besov and Besov-Orlicz spaces of dominating mixed smoothness, as well as Besov spaces of logarithmic dominating mixed smoothness (which surprisingly differ from  those  introduced by Triebel \cite{Triebel10}). Finally, we also propose new function spaces, which are strictly smaller than the best known spaces so far, where the paths of the Brownian sheet lie in almost surely. 

The structure of the paper is as follows. Section \ref{sec:d1} treats the univariate Wiener process. We first present   L\'evy's decomposition of its paths into the Faber system (Theorem \ref{thm:LevyDecomp}). Then (in Section \ref{sec:Faber}) we review the necessary notation from the area of function spaces. In Section \ref{sub:2.3} we merge these two subjects and re-prove the results of L\'evy and Ciesielski, giving an alternative proof for the Besov-Orlicz space $B^{1/2}_{\Phi_2,\infty}([0,1])$ in Section \ref{sub:2.4}. The new function spaces, where the paths of the Wiener process lie in almost surely, are then investigated in Section \ref{sec:2.5}. Section \ref{sec:3} studies the regularity of Brownian sheets and follows essentially the same pattern. After reviewing the necessary tools in the  multivariate setting (L\'evy's decomposition, multivariate Faber systems, function spaces of dominating mixed smoothness) we re-prove the results of \cite{Kamont2} and sketch the new function spaces, in which one can find almost all paths of the Brownian sheet. Finally, to make the exposition self-contained, Section \ref{sec:app} collects basic properties of random variables (including Gaussian variables, their absolute values and the integrated absolute Wiener process). We also collect some facts about Orlicz spaces.

\section{Regularity of Brownian paths}\label{sec:d1}

In this section we discuss the regularity of the sample paths of the classical Wiener process. Our approach is based on the decomposition, which can be traced back to L\'evy \cite{Levy}. Essentially, it gives a decomposition of the Wiener paths into the Faber system of shifted and dilated hat functions, with the coefficients given by independent Gaussian variables. This, together with characterizations of various function spaces in terms of the Faber system, will allow us to re-prove the classical results of L\'evy and Ciesielski, as well as to define new function spaces, where the sample paths of the Wiener process lie in almost surely.

In our work (as it is common in the literature) the notions of  Wiener process and  Brownian motion are used as synonyms, which both refer to the following definition.
\begin{dfn}\label{def:BM} A real-valued random process $W=(W_t)_{t\ge 0}$ is called  Wiener process (or Brownian motion) if it satisfies
\begin{enumerate}
\item[1)] $W_0=0$;
\item[2)] $W$ has almost surely continuous paths, i.e., $W_t$ is almost surely continuous in $t$;
\item[3)] $W$ has independent increments, i.e., if $0\le t_0<t_1<\dots<t_n$, then $W_{t_n}-W_{t_{n-1}}, W_{t_{n-1}}-W_{t_{n-2}},\dots,W_{t_1}-W_{t_0}$ are independent {random} variables;
\item[4)] $W$ has Gaussian increments, i.e., $W_t-W_s\sim {\mathcal N}(0,t-s)$ for $0\le s\le t.$
\end{enumerate}
\end{dfn}

Let the random variables $(W_t)_{t\ge 0}$ be defined on the common probability space $(\Omega, {\mathcal F},{\mathbb P})$. Then, for every $\omega\in\Omega$ fixed, we call the mapping $t\to W_t(\omega)$ a Brownian path (or Wiener path).

\subsection{L\'evy's decomposition of Brownian paths}\label{subsec:21}

We now present L\'evy's representation of Brownian motion,
which is essentially a dyadic decomposition of the paths of Brownian motion into a series of piecewise linear functions with random coefficients.
Although much of this idea applies to general continuous functions on any closed interval,
we restrict ourselves to $(W_t)_{t\in I}$, where $I=[0,1]$.

For every $j\in\N_0$ we construct a (random) continuous function $W_j(t)$, which
is piecewise linear on all dyadic intervals
\begin{equation*}
I_{j,m}=\left[\frac{m}{2^j},\frac{m+1}{2^j}\right],\quad m\in\{0,\dots,2^j-1\}
\end{equation*}
and which coincides with a given path $W_t$ at their endpoints
\[
t_{j,m}=\frac{m}{2^j}, \quad m\in\{0,\dots,2^j\}.
\]

For $j=0$, we put $W_0(t)=W_1\cdot t$ and observe that $W_0(0)=W_0=0$ and $W_0(1)=W_1$, i.e.,
$W_0(t)$ coincides with $W_t$ for $t=t_{0,0}=0$ and $t=t_{0,1}=1.$

For $j=1$, we are looking for a continuous function $W_1(t)$, which would coincide with $W_t$
not only in $t_{1,0}=t_{0,0}=0$ and $t_{1,2}=t_{0,1}=1$, but also in $t_{1,1}=1/2.$
For this sake, we add to $W_0(t)$ a suitable multiple of a continuous function $v(t)$,
which vanishes at $t=0$ and $t=1$ and is linear on $I_{1,0}=[0,1/2]$ as well as on $I_{1,1}=[1/2,1]$.
Therefore, $v(t)$ is the usual hat function supported in $I$, i.e.,
\begin{equation*}
v(t)=\begin{cases} 2t\quad &\text{if}\ 0\le t<\frac{1}{2},\\
2(1-t)\quad &\text{if}\ \frac{1}{2}\le t<1,\\
0&\text{otherwise}
\end{cases}
\end{equation*}
and we put
\[
W_1(t)=W_0(t)+(W_{1/2}-W_0(1/2))v(t),\quad t\in[0,1].
\]

We proceed further inductively. Let $j\in\N$ be fixed and let us assume that 
$W_0(t),\dots,W_j(t)$ were already constructed.
Then $W_t-W_j(t)$ vanishes at $t_{j,m}$ for all $m\in\{0,1,\dots,2^j\}.$
For $m\in\{0,1,\dots,2^j-1\}$, we define $W_{j+1}(t)$ for $t\in I_{j,m}$ by adding to $W_{j}(t)$ a continuous piecewise linear function with support in $I_{j,m}$
to ensure that $W_{j+1}(t)=W_t$ also in the middle point of $I_{j,m}$,
i.e., in $t_{j+1,2m+1}=\frac{2m+1}{2^{j+1}}=\frac{m}{2^j}+\frac{1}{2^{j+1}}$.
Hence, we need to add to $W_{j}(t)$ a multiple of the hat function $v_{j,m}(t)=v(2^j(t-t_{j,m}))$ with support in $I_{j,m}$ (see Figure \ref{fig:vjm})
\begin{equation}\label{eq:vjm}
v_{j,m}(t)=\begin{cases} 2^{j+1}(t-2^{-j}m)\quad &\text{if}\ 2^{-j}m\le t<2^{-j}m+2^{-j-1},\\
2^{j+1}(2^{-j}(m+1)-t)\quad &\text{if}\ 2^{-j}m+2^{-j-1}\le t<2^{-j}(m+1),\\
0&\text{otherwise}.
\end{cases}
\end{equation}
\begin{figure}[h]
\center{\includegraphics[width=6cm]{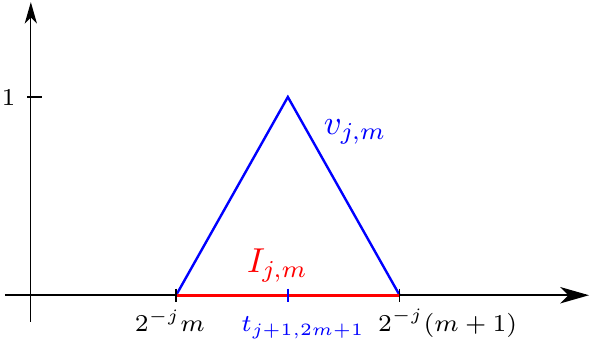}}
\caption{Hat function $v_{j,m}$}
\label{fig:vjm}
\end{figure}\\
We repeat this procedure for every $m\in\{0,1,\dots,2^j-1\}$ and obtain
\begin{equation}\label{eq:Wiener_decomp1}
W_{j+1}(t):=W_j(t)+\sum_{m=0}^{2^j-1} \left\{W_{t_{j+1,2m+1}}-W_j(t_{j+1,2m+1})\right\}v_{j,m}(t),\quad t\in[0,1].
\end{equation}

\begin{figure}[h]
\center{\includegraphics[width=12cm]{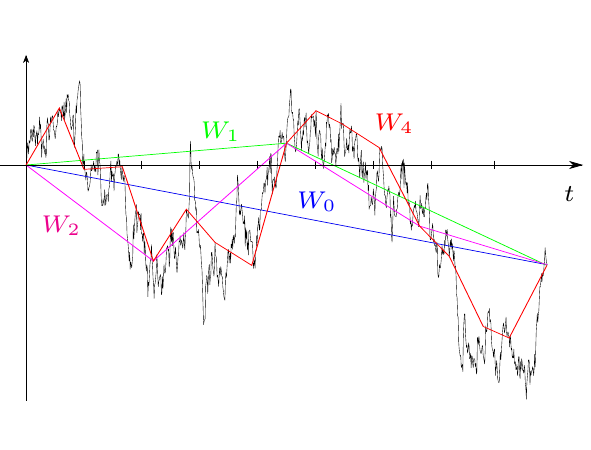}}
\caption{Piecewise linear functions $W_j(t)$ approximating $W_t$}
\end{figure}

The main disadvantage of \eqref{eq:Wiener_decomp1} is that the coefficients in the sum over $m$ involve both the values
of the Wiener path $W_t$ as well as the values of its approximation $W_j(t)$. Therefore, we rewrite it in such a form, that
only the values of $W_t$ at dyadic points get used.

By its construction, $W_j(t)$ is linear on every $I_{j,m}$ and coincides with $W_t$ at its endpoints and, therefore, we may also write it as
\[
W_j(t)=W_{\frac{m}{2^j}}+\Bigl(t-\frac{m}{2^j}\Bigr)\cdot 2^j\cdot \Bigl[W_{\frac{m+1}{2^j}}-W_{\frac{m}{2^j}}\Bigr],\quad t\in I_{j,m}.
\]
This allows us to rewrite the coefficients of \eqref{eq:Wiener_decomp1} as
\begin{align*}
W_{t_{j+1,2m+1}}-W_j(t_{j+1,2m+1})&=W_{\frac{2m+1}{2^{j+1}}}-W_j\Bigl(\frac{2m+1}{2^{j+1}}\Bigr)\\
&=W_{\frac{2m+1}{2^{j+1}}}-\Bigl(W_{\frac{m}{2^j}}+\frac{1}{2^{j+1}}\cdot 2^j\cdot \Bigl[W_{\frac{m+1}{2^j}}-W_{\frac{m}{2^j}}\Bigr]\Bigr)\\
&=-\frac{1}{2}\Bigl(W_{\frac{2m+2}{2^{j+1}}}-2W_{\frac{2m+1}{2^{j+1}}}+W_{\frac{2m}{2^{j+1}}}\Bigr)\\
&=-\frac{1}{2}(\Delta^2_{2^{-j-1}}W)\Bigl(\frac{2m}{2^{j+1}}\Bigr),
\end{align*}
where
\[
(\Delta^2_hf)(x)=f(x+2h)-2f(x+h)+f(x)=\Bigl(f(x+2h)-f(x+h)\Bigr)-\Bigl(f(x+h)-f(x)\Bigr)
\]
are the second order differences of a function $f$. Together with \eqref{eq:Wiener_decomp1} this leads to
\begin{equation}\label{eq:Wiener_decomp2}
W_{j+1}(t):=W_j(t)-\frac{1}{2}\sum_{m=0}^{2^j-1} (\Delta^2_{2^{-j-1}}W)\Bigl(\frac{2m}{2^{j+1}}\Bigr)v_{j,m}(t),\quad t\in[0,1].
\end{equation}

The reader may notice that all what we did so far, including \eqref{eq:Wiener_decomp2}, applies
to general continuous functions on $I$. We summarize this in the following theorem (and refer to \cite[Theorem 2.1]{Triebel10} for a detailed proof
and to \cite{Fab09,Haar10} for historic sources).
\begin{thm}\label{thm:triebel1} Let $f\in C(I)$. Then
\begin{equation}\label{eq:triebel1}
f(t)=f(0)\cdot (1-t)+f(1)\cdot t-\frac{1}{2}\sum_{j=0}^\infty \sum_{m=0}^{2^j-1}(\Delta^2_{2^{-j-1}}f)(2^{-j}m)v_{j,m}(t)
\end{equation}
for every $0\le t\le 1$ and the series converges uniformly on $I$.
\end{thm}

To transform \eqref{eq:Wiener_decomp2} into a series representation of $W_t$, we note that the variables
$W_{\frac{2m+2}{2^{j+1}}}-W_{\frac{2m+1}{2^{j+1}}}$ and $W_{\frac{2m+1}{2^{j+1}}}-W_{\frac{2m}{2^{j+1}}}$
are independent and have  distribution ${\mathcal N}(0,2^{-(j+1)})$.
Hence, by the 2-stability of the normal distribution,  cf. Lemma \ref{2-stability}, 
$W_{t_{j+1,2m+1}}-W_j(t_{j+1,2m+1})=-\frac{1}{2}(\Delta^2_{2^{-j-1}}W)\Bigl(\frac{2m}{2^{j+1}}\Bigr)$ is normally distributed with mean zero and variance $2^{-(j+2)}$.

We can therefore rewrite \eqref{eq:Wiener_decomp2} as
$$
W_{j+1}(t)=W_j(t)+\sum_{m=0}^{2^j-1} 2^{-(j+2)/2}\xi_{j,m}v_{j,m}(t),\quad t\in[0,1],
$$
where $\xi_{j,m}$ are standard normal variables.
An explicit formula for $W_t$ can then be obtained by noting that the series
\begin{equation}\label{eq:Levy_series}
\sum_{j=0}^{\infty}\Bigl(W_{j+1}(t)-W_j(t)\Bigr)
\end{equation}
converges almost surely uniformly to \[
W_t-W_0(t)=W_t-W_1\cdot t\qquad  \text{with}\qquad  W_1= \xi_{-1}\sim {\mathcal N}(0,1).
\]
This follows by the tail bound for normal variables, cf. Lemma \ref{lem:tail_bounds}, and a straightforward union bound, which give for every real $A>1$ that 
\begin{align*}
\P\left(\exists j\in\N_0: \|W_{j+1}-W_j\|_\infty>A\cdot 2^{-j/4}\right)
&\le \sum_{j=0}^\infty \sum_{m=0}^{2^j-1}\P(|\xi_{j,m}|>2A\cdot 2^{j/4})\\
&\le \sum_{j=0}^\infty 2^j\exp(-2A^2\cdot 2^{j/2}). 
\end{align*}
If $A$ goes to infinity, the last sum tends to zero and, therefore, the probability that $\|W_{j+1}-W_j\|_\infty\le A\cdot 2^{-j/4}$ for all $j\in\N_0$ grows to one. This ensures that \eqref{eq:Levy_series} converges uniformly almost surely.


This yields that we have almost surely
\begin{equation}\label{eq:Levy01}
W_t=\xi_{-1}\cdot t+\sum_{j=0}^\infty \sum_{m=0}^{2^j-1} 2^{-(j+2)/2}\xi_{j,m}v_{j,m}(t),\quad t\in[0,1]. 
\end{equation}

As the last step, we need to complement \eqref{eq:Levy01} by the crucial observation that the random variables $\{\xi_{-1}\}\cup\{\xi_{j,m}, j\in\N_0, 0\le m\le 2^j-1\}$ are independent.
For that sake, let $\xi_{j_1,m_1},\dots,\xi_{j_N,m_N}$ be fixed and let us put
\[
J=\max(j_1,\dots,j_N).
\]
We collect the independent Gaussian variables
\[
W^l=W_{\frac{l+1}{2^{J+1}}}-W_{\frac{l}{2^{J+1}}},\quad l=0,1,\dots,2^{J+1}-1,
\]
into a vector $\widetilde W^J=(W^0,\dots,W^{2^{J+1}-1})^T.$
Using this notation, we observe that
\begin{align*}
\xi_{-1}&=W_1-W_0=\sum_{l=0}^{2^{J+1}-1}W^l=\langle(1,\dots,1)^T,\widetilde W^J\rangle.
\intertext{Similarly, for every $0\,\le\,j\le J$ and $m\in\{0,\dots,2^{j}-1\}$, we get}
-2^{-j/2}\xi_{j,m}&=\Bigl(W_{\frac{2m+2}{2^{j+1}}}-W_{\frac{2m+1}{2^{j+1}}}\Bigr)-\Bigl(W_{\frac{2m+1}{2^{j+1}}}-W_{\frac{2m}{2^{j+1}}}\Bigr)\\
&=\sum_{l=0}^{2^{J-j}-1} W^{(2m+1)2^{J-j}+l}-\sum_{l=0}^{2^{J-j}-1}W^{2m\cdot 2^{J-j}+l}=\langle h^J_{j,m},\widetilde W^J\rangle,
\end{align*}
where
\begin{equation}\label{eq:def_h}
(h^J_{j,m})_{l }=\begin{cases}+1\quad &\text{if}\ (2m+1)2^{J-j}\le  l <(2m+2)2^{J-j},\\
-1\quad &\text{if}\ 2m\cdot 2^{J-j}\le  l <(2m+1)2^{J-j},\\
0\quad &\text{otherwise}
\end{cases}
\end{equation}
for $0\le  l \le 2^{J+1}-1$.
The independence of $\xi_{j_1,m_1},\dots,\xi_{j_N,m_N}$ now follows from 
the orthogonality of the vectors $(h_{j_i,m_i}^J)_{i=1}^N$, cf. Lemma \ref{2-stability}. This leads to the following representation.
\begin{thm}\label{thm:LevyDecomp}
Let $(W_t)_{t\in I}$ be the Brownian motion according to Definition \ref{def:BM}. If $v_{j,m}$ denotes the Faber system according to \eqref{eq:vjm}, then almost surely it holds
\begin{align*}
W_t=\xi_{-1}\cdot t+\sum_{j=0}^\infty \sum_{m=0}^{2^j-1} 2^{-(j+2)/2}\xi_{j,m}v_{j,m}(t)\quad\text{ for } t\in[0,1],
\end{align*}
where $\{\xi_{-1}\}\cup\{\xi_{j,m}: j\in\N_0, 0\le m\le 2^j-1\}$ are independent $\mathcal{N}(0,1)$ random variables and the series converges uniformly on $I$.
\end{thm}

\subsection{Function spaces and Faber systems}\label{sec:Faber}

As already explained in Section \ref{sec:Into}, the description of the regularity of paths of Brownian motion will be given in different
scales of function spaces of Besov, H\"older, and Orlicz type.
In the sequel, we try to give in brief the basic definitions and characterizations of these spaces.

We assume, that the reader is familiar with the spaces of complex-valued continuous functions $C(\R)$
and $C(I)$ as well as with the Lebesgue spaces of integrable functions $L_p(\R)$ and $L_p(I)$.
For Lebesgue spaces we simplify the notation by writing the norms
\begin{align}
    \|f\|_p:=\norm{f}{L_p}.\label{simplenorm}
\end{align}
The domain $I$ or $\R$ of the function $f$ in \eqref{simplenorm} should always be clear from the context.

For any $0<p\le \infty$ and any $f\in L_p(\R)$, we denote by
\begin{equation}\label{differences}
(\Delta^1_h f)(x)=f(x+h)-f(x), \qquad (\Delta^{M+1}_h f)=\Delta_h^1(\Delta_h^M f),
\end{equation}
the usual first-order and higher-order differences (as already briefly mentioned in the previous section), where $x\in \R$, $h\in \R$ and $M\in \mathbb{N}$. We start with the definition of Besov spaces. These spaces have a long history and many of their aspects were studied in the last decades, cf. \cite{Peetre, SiRu, TrFS1}.

\begin{dfn}
\begin{enumerate}
\item[(i)] Let $s>0$ and $0<p,q\leq \infty$. Then the Besov space $B^s_{p,q}(\R)$ is the collection of all $f\in L_p(\R)$ such that for  $M=\lfloor s\rfloor$+1, 
\[
\|f|B^s_{p,q}(\R)\|=\|f\|_p+\left(\int_0^1 t^{-sq}\sup_{|h|\leq t} \|\Delta^M_hf\|_p^q\frac{dt}{t}\right)^{1/q}<\infty. 
\]
Here, $\lfloor s\rfloor$ is the greatest integer less than or equal to $s$.
\item[(ii)] Besov spaces on the interval $I=[0,1]\subset \R$ are defined via restriction, i.e., 
\begin{equation}\label{restriction}
B^s_{p,q}(I):=\left\{f\in L_p(I): \ f=g\big|_{I}\ \text{ for some }\ g\in B^s_{p,q}(\R)\right\}, 
\end{equation}
normed by 
\[
\|f|B^s_{p,q}(I)\|=\inf \|g|B^s_{p,q}(\R)\|,
\]
where the infimum is taken over all $g\in B^s_{p,q}(\R)$ with $g\big|_{I}=f$.
\end{enumerate}
\end{dfn}

Our approach to regularity of Brownian paths is based on the close connection
between Besov spaces (and other function spaces) and the decompositions in the Faber system.
The Faber system on the interval $I=[0,1]$ is the collection of functions
\[
\{v_0,v_1,v_{j,m}:j\in\N_0, m=0,\dots,2^{j}-1\},
\]
where
\[
v_0(x)=1-x,\quad v_1(x)=x,\quad x\in I
\]
and $v_{j,m}$ is defined by \eqref{eq:vjm} for $j\ge 0$. 
Then Theorem 2.1 of \cite{Triebel10} (cf. Theorem \ref{thm:triebel1}) shows that the Faber system
is a (conditional) basis of $C(I)$ and that every $f\in C(I)$ can be written as
\begin{equation*}
f(x)=f(0)v_0(x)+f(1)v_1(x)-\frac{1}{2}\sum_{j=0}^\infty \sum_{m=0}^{2^j-1}(\Delta^2_{2^{-j-1}}f)(2^{-j}m)v_{j,m}(x),\quad x\in I.
\end{equation*}

Furthermore, concerning the decomposition of Besov spaces $B^s_{p,q}(I)$ with the above  Faber system, we recall Theorem 3.1 in \cite[page 126]{Triebel10}.   

\begin{thm}\label{thm:faber_s} Let  $0<p,q\le\infty$ and\\
\begin{minipage}{0.6\textwidth}
\begin{equation*}
\frac{1}{p}<s<1+\min\Bigl(\frac{1}{p},1\Bigr)
\end{equation*}
be the admissible range for $s$ as illustrated in the figure aside.  Then the sum
\begin{equation}\label{eq:Triebel01}
f=\mu_0v_0+\mu_1v_1+\sum_{j=0}^\infty \sum_{m=0}^{2^j-1}\mu_{j,m}2^{-js}v_{j,m}
\end{equation}
\end{minipage}\hfill \begin{minipage}{0.3\textwidth}
\includegraphics[width=4cm]{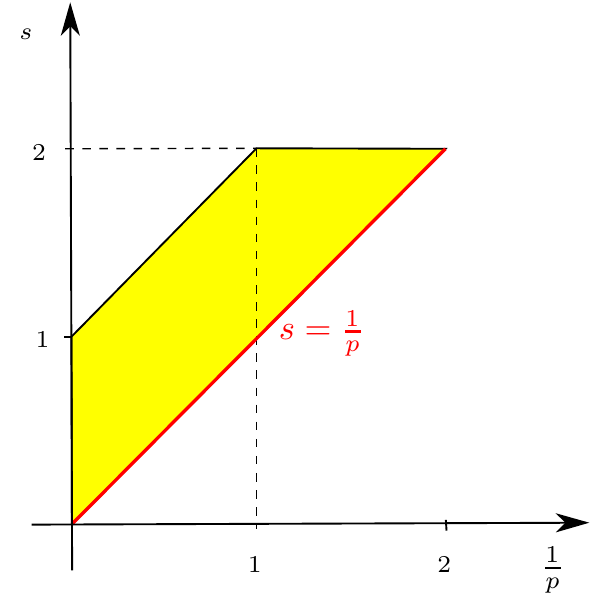}
\end{minipage}\\
with $\mu_0=\mu_0(f)=f(0)$, $\mu_1=\mu_1(f)=f(1)$ and $$\mu_{j,m}=\mu_{j,m}(f)=-2^{js-1}(\Delta^2_{2^{-j-1}}f)(2^{-j}m)$$
lies in $B^s_{p,q}(I)$ if, and only if,
$$
\|\mu|b^+_{p,q}(I)\|:=|\mu_0|+|\mu_1|+\biggl(\sum_{j=0}^\infty \Bigl(\sum_{m=0}^{2^j-1}2^{-j}|\mu_{j,m}|^p\Bigr)^{q/p}\biggr)^{1/q}<\infty.
$$
\end{thm}

\subsubsection{Function spaces of logarithmic smoothness}
As pointed out already in the Introduction, one can obtain finer  descriptions of the regularity properties of Brownian paths,
if one uses different (and more sophisticated) scales of function spaces.

Therefore, we now introduce   the so-called function spaces of logarithmic smoothness,
cf. \cite{DomTikh, FaLeo, KaLi, Moura}.
Here again we rely on the exposition and results from \cite{Triebel10}. In particular, function spaces of logarithmic smoothness are a special case of Besov and Triebel-Lizorkin spaces of generalized smoothness, where the smoothness factor
\[
    2^{js}\qquad \text{gets replaced by}\qquad 2^{js}(1+j)^{-\alpha}
\]
with $s,\alpha\in\R$. Following \cite[Proposition 1.7.4]{Triebel10} we can define function spaces of logarithmic smoothness by differences as follows.
\begin{dfn}
Let $0<p,q\leq \infty$, $s,\alpha\in\R$, and $M=\lfloor s\rfloor +1$ with
\begin{align*}
    s>\max\left(\frac1p,1\right)-1.
\end{align*}
Then the logarithmic space $B^{s,\alpha}_{p,q}(\R)$ contains all functions $f\in L_p(\R)$ with
\begin{align*}
    \norm{f}{B^{s,\alpha}_{p,q}(\R)}=\|f\|_p+\left(\int_0^1t^{-sq}(1+|\log t|)^{-\alpha q}\sup_{0<h<t}\|\Delta_h^M f\|_p^q\frac{d t}{t}\right)^{1/q}<\infty.
\end{align*}
\end{dfn}
Note that in comparison with \cite{Triebel10} we have replaced $\alpha$ by $-\alpha$ in the definition of the logarithmic space $B^{s,\alpha}_{p,q}(\R)$ leading to some minor adaptations in the theorem below.
Following \cite[Theorem 3.30]{Triebel10} we have also a characterization in terms of the Faber system of the spaces $B^{s,\alpha}_{p,q}(I)$ in which we are interested here. The restriction from $\R$ to $I$ is done in the same way as described in Subsection \ref{sec:Faber}, cf. \eqref{restriction}.
\begin{thm}\label{thm:faber_log}
Let $0<p,q\leq\infty$ and $s,\alpha\in\R$  with
\begin{align*}
    \frac1p<s<1+\min\left(\frac1p,1\right).
\end{align*}
Then $f\in L_p(I)$ belongs to $B^{s,\alpha}_{p,q}(I)$ if, and only if, it can be represented as 
\begin{align}\label{eq:log_faber}
   f=\mu_0v_0+\mu_1v_1+\sum_{j=0}^\infty \sum_{m=0}^{2^j-1}\mu_{j,m}2^{-js}v_{j,m},
\end{align}
where $\mu_{j,m}\in b_{p,q}^{+,\alpha}(I)$, i.e., 
\begin{align*}
    \|\mu|b^{+,\alpha}_{p,q}(I)\|:=|\mu_0|+|\mu_1|+\biggl(\sum_{j=0}^\infty(1+j)^{-\alpha q} \Bigl(\sum_{m=0}^{2^j-1}2^{-j}|\mu_{j,m}|^p\Bigr)^{q/p}\biggr)^{1/q}<\infty.
\end{align*}
Here the sum in \eqref{eq:log_faber} converges unconditionally  in $B^\sigma_{p,q}(I)$ for $\sigma<s$ and in  $C(I)$. Moreover,  the representation is unique with
\begin{align*}
    \mu_0&=\mu_0(f)=f(0),\,\mu_1=\mu_1(f)=f(1) \text{ and}\\
    \mu_{j,m}&=\mu_{j,m}(f)=-2^{js-1}(\Delta^2_{2^{-j-1}}f)(2^{-j}m)
\end{align*}
where $j\in\N_0$ and $m=0,\dotsc,2^j-1$.
\end{thm}
\begin{rem}
In the theory of function spaces there are two different approaches on how to deal with the normalization factors appearing in characterizations with building blocks such as
atoms, wavelets or Faber functions (as it is our case here). The first approach tries to take the same building blocks for any function space $B^s_{p,q}$ or $B^{s,\alpha}_{p,q}$
independent from the chosen smoothness parameters $s,\alpha$. This results in the adaption of the corresponding sequence spaces in $b^s_{p,q}$ or $b^{s,\alpha}_{p,q}$,
where now these parameters play a role.

The second approach changes the definition of the building blocks according to the smoothness parameters $s$ and $\alpha$. The consequence is that then the
corresponding sequence spaces $b_{p,q}$ are independent on $s$ and $\alpha$.
In this work, we  always use the same decomposition of the function, which corresponds to  L\'evy's decomposition \eqref{eq:Levy01}.
For that reason, we prefer to work with \eqref{eq:Triebel01} and \eqref{eq:log_faber} in the above theorems. This results in Theorem \ref{thm:faber_s}, where 
  the sequence spaces $b^+_{p,q}$ corresponding to $B^s_{p,q}(I)$ are independent of the chosen smoothness $s$.
Furthermore in Theorem \ref{thm:faber_log} we only include the logarithmic smoothness parameter within
the sequence space norm. To that end, now $b^{+,\alpha}_{p,q}(I)$ corresponds to $B^{s,\alpha}_{p,q}(I)$, which is still independent on $s$ but depends on $\alpha$.
\end{rem}

\subsubsection{Besov-Orlicz spaces }

We replace the Lebesgue norm $\|\cdot\|_\infty$ in the interesting boundary case $p=\infty$ in the definition of the Besov spaces by the Orlicz-norm $\|\cdot\|_{\Phi_2}$, which is given by the following Young function
\begin{equation}\label{eq:Orl01}
\Phi_2(t)=\exp(t^2)-1\quad \text{for}\quad t>0.
\end{equation}
We will define for dimension $d\in\N$ the Orlicz space $L_{\Phi_2}([0,1]^d)$ as the collection of all measurable functions on $[0,1]^d$ with
\begin{equation}\label{eq:Orl02}
\|f\|_{{\Phi_2}}:=\inf\left\{\lambda>0: \int_{[0,1]^d}\Phi_2\biggl(\frac{|f(t)|}{\lambda}\biggr)dt\le 1\right\}<\infty.
\end{equation}

By Theorem \ref{thm:Orlicz} and Theorem \ref{thm:phi2}, $\|f\|_{{\Phi_2}}$ is also equivalent to
\begin{align*}
\sup_{0<t<1}\frac{f^*(t)}{\sqrt{\log(1/t)+1}}\qquad\text{and}\qquad \sup_{p\geq1}\frac{\|f\|_p}{\sqrt{p}},
\end{align*}
where
\begin{align*}
f^*(t)=\inf\left\{s\in[0,1]:|\{r\in[0,1]^d:|f(r)|>s\}|\leq t\right\}
\end{align*}
is the non-increasing rearrangement of $f$. This Orlicz norm and its equivalent expressions play a fundamental role in the characterization of sub-gaussian random variables \cite{Kahane,Ver}.

Now we are going to define the Besov-Orlicz spaces $B^{1/2}_{\Phi_2,\infty}(I)$ directly via the decompositions with the Faber system. The sequence space norm is a direct adaptation from the sequence spaces $b^+_{p,q}(I)$ from Theorem \ref{thm:faber_s},
where the $L_p(I)$ is now replaced by the Orlicz norm $L_{\Phi_2}(I)$, and the characterization from Theorem \ref{thm:phi2}. We also refer   to \cite[Theorem III.8]{CKR93}, where a characterization of these spaces with the help of the Faber system is given and to \cite{PiSi}, where one can find an alternative approach to Besov-Orlicz spaces.

\begin{dfn}\label{def:orlicz_besov}
\begin{itemize}
\item[(i)] The sequence space $b^{+}_{\Phi_2,\infty}$ is the collection of all sequences 
\[
\{\mu=(\mu_{j,m}): \ j\in \N_0, \ m=0,\ldots, 2^j-1\}
\]
such that 
\begin{align*}
\norm{\mu}{b^{+}_{\Phi_2,\infty}(I)}&:=\sup_{j\in\N_0}\left\|
\sum_{m=0}^{2^j-1}\mu_{j,m}\chi_{j,m}\right\|_{\Phi_2}\approx\sup_{j\in\N_0}\sup_{p\ge 1}\frac{1}{\sqrt{p}} \left\|\sum_{m=0}^{2^j-1}\mu_{j,m}\chi_{j,m}\right\|_p\\
&=\sup_{j\in\N_0}\sup_{p\ge 1}\frac{1}{\sqrt{p}}
\Bigl(\sum_{m=0}^{2^j-1}2^{-j}|\mu_{j,m}|^p\Bigr)^{1/p}<\infty,
\end{align*}
where $\chi_{j,m}$ is the characteristic function of $I_{j,m}$.
\item[(ii)] The function space $B^{1/2}_{{\Phi_2},\infty}(I)$ is the collection of all $f\in C(I)$ such that the coefficients of its representation 
\begin{equation}\label{eq:decomp_Faber1}
f(x)=\lambda_0v_0(x)+\lambda_1v_1(x)+\sum_{j=0}^\infty \sum_{m=0}^{2^j-1}2^{-j/2} \lambda_{j,m}v_{j,m}(x),\quad x\in I,
\end{equation}
satisfy $\norm{\lambda}{b^{+}_{\Phi_2,\infty}(I)}<\infty$.
\end{itemize}
\end{dfn} 

\subsection{Results of L\'evy and Ciesielski}\label{sub:2.3}

In this section we present proofs of the results of L\'evy \cite{Levy} and Ciesielski \cite{Ciesiel93} concerning the regularity of Wiener paths in certain function spaces.
Our main aim is to show that they both follow quite directly from L\'evy's decomposition of Wiener paths into the Faber system \eqref{eq:Levy01} combined
with the characterization of the corresponding function spaces via the Faber system.
For that sake, we summarize Theorem \ref{thm:faber_s}, Theorem \ref{thm:faber_log}, and Definition \ref{def:orlicz_besov}, which state
the conditions on the coefficients guaranteeing that a function belongs to the function spaces considered by L\'evy and Ciesielski.
In particular, we choose a formulation, which corresponds directly to \eqref{eq:Levy01}.

\begin{thm}\label{thm:function_spaces}
Consider a function $f\in C(I)$ with the representation 
\begin{equation*}
f(x)=\lambda_0v_0(x)+\lambda_1v_1(x)+\sum_{j=0}^\infty \sum_{m=0}^{2^j-1} 2^{-\frac{j+2}{2}} \lambda_{j,m}v_{j,m}(x),\quad x\in I,
\end{equation*}
where $\{v_0,v_1, v_{j,m}: \,j\in \N_0,\, m=0,\ldots, 2^j-1\}$ 
denotes the Faber system on the interval $I=[0,1]$. 
\begin{enumerate}
\item[(i)] $f$ belongs to $B^{1/2,1/2}_{\infty,\infty}(I)$ if, and only if
\begin{equation}\label{eq:char_1}
\sup_{j\in\N}\frac{1}{\sqrt{j}}\sup_{m=0,\dots,2^j-1}|\lambda_{j,m}| < \infty.
\end{equation}
\item[(ii)] Let $1\le p<\infty$. Then $f$ belongs to $B^{1/2}_{p,\infty}(I)$ if, and only if
\begin{equation}\label{eq:char_2}
\sup_{j\in\N}\Bigl(\sum_{m=0}^{2^j-1} 2^{-j}|\lambda_{j,m}|^p\Bigr)^{1/p} < \infty.
\end{equation}
\item[(iii)] $f$ belongs to $B^{1/2}_{{\Phi_2},\infty}(I)$ if, and only if
\begin{equation}\label{eq:char_3}
\sup_{j\in\N}\sup_{p\ge 1}\frac{1}{\sqrt p}\Bigl(\sum_{m=0}^{2^j-1} 2^{-j}|\lambda_{j,m}|^p\Bigr)^{1/p} < \infty.
\end{equation}
\end{enumerate}
\end{thm}

\begin{enumerate}
\item {\bf Wiener paths belong to $B^{1/2,1/2}_{\infty,\infty}(I)$ almost surely (L\'evy \cite{Levy})}

Comparing \eqref{eq:Levy01} with \eqref{eq:char_1}, it is enough to show that
\begin{equation}\label{eq:cond_Levy}
\sup_{j\in\N}\frac{1}{\sqrt{j}}\sup_{m=0,\dots,2^j-1}|\xi_{j,m}|<\infty\quad\text{almost surely},
\end{equation}
where $\{\xi_{j,m}:j\in\N, m=0,\dots, 2^j-1\}$ are independent standard Gaussian variables.

To prove \eqref{eq:cond_Levy}, we denote by $A^N_{j}$ the event when $\sup_{m=0,\dots,2^{j}-1}|\xi_{j,m}|\ge N\sqrt{j}$ and estimate
\begin{align}\label{eq:AtN_upper1}
{\mathbb P}(A^N_{j})\le 2^j{\mathbb P}(|\omega|\ge N\sqrt{j})\le 2^j e^{-N^2j/2}, 
\end{align}
where we used the estimate from Lemma \ref{lem:tail_bounds} (i). 
Using an estimate, which resembles the approach of the Borel-Cantelli lemma and which we shall use frequently later on,
we obtain for every positive integer $N_0$,
\begin{align}
\notag{\mathbb P}\Bigl(\sup_{j\in\N}\frac{1}{\sqrt{j}}\sup_{m=0,\dots,2^j-1}|\xi_{j,m}|=\infty\Bigr)
&={\mathbb P}\Bigl(\bigcap_{N=1}^\infty\bigcup_{j=1}^\infty A^N_{j}\Bigr)\le 
{\mathbb P}\Bigl(\bigcup_{j=1}^\infty A^{N_0}_{j}\Bigr)\\
\label{eq:sim1}&\le \sum_{j=1}^\infty {\mathbb P}\bigl(A^{N_0}_{j}\bigr)
\le \sum_{j=1}^\infty 2^j e^{-N_0^2 j/2}.
\end{align}
As the last expression tends to zero if $N_0\to\infty,$ this finishes the proof of \eqref{eq:cond_Levy}.

\item {\bf Wiener paths belong to $B^{1/2}_{p,\infty}(I)$ almost surely (Ciesielski \cite{Ciesiel91})}

We restrict ourselves to $2<p<\infty$, which allows us to use Theorem \ref{thm:faber_s}. The smaller values of $p$ are then covered by the monotonicity of Besov spaces on domains with respect to the integrability parameter $p$. Again, by \eqref{eq:Levy01} and \eqref{eq:char_2}, it is enough to prove that
\begin{equation}\label{eq:char_2'}
\sup_{j\in\N}\Bigl(\sum_{m=0}^{2^j-1} 2^{-j}|\xi_{j,m}|^p\Bigr)^{1/p} < \infty\quad\text{almost surely for every}\quad 2< p<\infty .
\end{equation}

Fix $2< p<\infty$ and let $\mu_{p}=\E\,|\omega|^p$ be the $p$th absolute moment of a standard Gaussian variable $\omega$.
This time, we denote for every $t>0$ and $j\in\N$ by $A^t_j$ the event that 
\[
\frac{1}{2^j}\sum_{m=0}^{2^j-1}|\xi_{j,m}|^{p}-\mu_{p}\ge t.
\]
Then, by Markov's inequality,
\begin{align}
\notag t^2 \P(A^t_j)&=t^2 \P\biggl(\Bigl(\frac{1}{2^j}\sum_{m=0}^{2^j-1}|\xi_{j,m}|^{p}-\mu_{p}\Bigr)^2\ge t^2\biggr)
\le \E \Bigl(\frac{1}{2^j}\sum_{m=0}^{2^j-1}|\xi_{j,m}|^{p}-\mu_{p}\Bigr)^2\\
\label{eq:Ajt}&=\E\, \frac{1}{2^{2j}}\sum_{m=0}^{2^j-1}|\xi_{j,m}|^{2p}+\E\, \frac{1}{2^{2j}}\sum_{m\not=n=0}^{2^j-1}|\xi_{j,m}|^{p}|\xi_{j,n}|^{p}
-\frac{2\mu_{p}}{2^j}\E\,\sum_{m=0}^{2^j-1}|\xi_{j,m}|^{p}+\mu_{p}^2\\
\notag &=\frac{\mu_{2p}}{2^j}+\frac{2^j(2^j-1)}{2^{2j}}\mu_{p}^2-\mu_{p}^2=\frac{\mu_{2p}-\mu_{p}^2}{2^j}.
\end{align}
Similarly to \eqref{eq:sim1}, we conclude, that for every $N_0\in\N$ it holds
\begin{align*}
\P\biggl(\sup_{j\in\N_0}\frac{1}{2^j}\sum_{m=0}^{2^j-1}|\xi_{j,m}|^{p}=\infty\biggr)&=\P\Bigl(\bigcap_{N=1}^\infty\bigcup_{j=0}^\infty A_j^N\Bigr)
\le \P\Bigl(\bigcup_{j=0}^\infty A_j^{N_0}\Bigr)\le \sum_{j=0}^\infty \P\bigl(A_j^{N_0}\bigr)\\
&\le \sum_{j=0}^\infty \frac{\mu_{2p}-\mu_p^2}{2^jN_0^2}=\frac{2(\mu_{2p}-\mu_p^2)}{N_0^2}.
\end{align*}
The last expression tends to zero if $N_0\to\infty$, which renders \eqref{eq:char_2'}.

\item {\bf Wiener paths belong to $B^{1/2}_{{\Phi_2},\infty}(I)$ almost surely (Ciesielski \cite{Ciesiel93})}

This time, we need to show that
\begin{equation}\label{eq:cond_Levy3}
\sup_{j\in\N}\sup_{p\ge 1}\frac{1}{\sqrt{p}}\Bigl(\frac{1}{2^j}\sum_{m=0}^{2^j-1}|\xi_{j,m}|^p\Bigr)^{1/p}<\infty\quad \text{almost surely}. 
\end{equation}
By monotonicity, it is enough to restrict the supremum over $p$ to the integer values $p\in\N.$
Furthermore, we only need to refine the analysis done before. Indeed, it follows directly from \eqref{eq:Ajt} that
\[
\P\biggl(\frac{1}{\sqrt{p}}\Bigl(\frac{1}{2^j}\sum_{m=0}^{2^j-1}|\xi_{j,m}|^p\Bigr)^{1/p}\ge \frac{(2t)^{1/p}}{\sqrt{p}}\biggr)
\le \P(A_j^t) \le \frac{\mu_{2p}-\mu_p^2}{2^jt^2}\quad\text{for every}\quad t>\mu_p.
\]
Hence, if $N>(2\mu_p)^{1/p}/\sqrt{p}$, and if $A_j^{N,p}$ denotes the event when
\[
\frac{1}{\sqrt{p}}\Bigl(\frac{1}{2^j}\sum_{m=0}^{2^j-1}|\xi_{j,m}|^p\Bigr)^{1/p}\ge N,
\]
then
\[
\P\bigl(A_j^{N,p}\bigr)
\le \frac{\mu_{2p}-\mu_p^2}{2^{j-2}(N\sqrt{p})^{2p}}\le \frac{\mu_{2p}}{2^{j-2}(N\sqrt{p})^{2p}}.
\]
Since $\displaystyle \mu_{2p}=\frac{2^{p}\Gamma\left(p+\frac{1}{2}\right)}{\sqrt{\pi}}\le 2^p\cdot \Gamma(p+1)=2^p\cdot p!\le (2p)^p$, we obtain for every $N_0$ large enough
\begin{align*}
\P\biggl(&\sup_{j\in\N}\sup_{p\in\N}\frac{1}{\sqrt{p}}\Bigl(\frac{1}{2^j}\sum_{m=0}^{2^j-1}|\xi_{j,m}|^p\Bigr)^{1/p}=\infty\biggr)
=\P\Bigl(\bigcap_{N=1}^\infty\bigcup_{j,p=1}^\infty A_j^{N,p}\Bigr)\le \P\Bigl(\bigcup_{j,p=1}^\infty A_j^{N_0,p}\Bigr)\\
&\le \sum_{j,p=1}^\infty \P\bigl(A_j^{N_0,p}\bigr)
\le \sum_{j,p=1}^\infty\frac{\mu_{2p}}{2^{j-2}(N_0\sqrt{p})^{2p}}\lesssim \sum_{p=1}^\infty\frac{2^p}{N_0^{2p}}.
\end{align*}
As the last expression tends to zero if $N_0\to\infty$, we obtain \eqref{eq:cond_Levy3}.
\end{enumerate}

\begin{rem}
Let us have a closer look at the relations of the different spaces we are dealing with in terms of  their embeddings.
We rely on the characterizations \eqref{eq:char_1}, \eqref{eq:char_2}, and \eqref{eq:char_3} collected in Theorem \ref{thm:function_spaces}.
Furthermore, one can show in the same way that a function $f$ with \eqref{eq:decomp_Faber1} belongs to the H\"older-Zygmund space
$B^{1/2}_{\infty,\infty}(I)$ if, and only if,
\begin{equation}\label{eq:char_4}
    \sup_{j\in\N_0}\sup_{m=0,\dots,2^{j}-1}|\lambda_{j,m}|<\infty.
\end{equation}
One can use \eqref{eq:char_4} to recover the well known fact that the Wiener paths almost surely do not belong to $B^{1/2}_{\infty,\infty}(I).$
Indeed, the supremum of $2^j$ independent standard Gaussian variables grows asymptotically as $\sqrt{j}$, cf. also Lemma \ref{lem:tail_bounds}(ii), which violates \eqref{eq:char_4}. Furthermore, \eqref{eq:char_2}, \eqref{eq:char_3}, and \eqref{eq:char_4} show that
\[
B^{1/2}_{\infty,\infty}(I)\hookrightarrow B^{1/2}_{{\Phi_2},\infty}(I)\hookrightarrow B^{1/2}_{p,\infty}(I).
\]

To compare the spaces $B^{1/2}_{{\Phi_2},\infty}(I)$ and $B^{1/2,1/2}_{\infty,\infty}(I)$ we estimate for every $j\in\N$
\begin{align*}
 \frac{1}{\sqrt{j}}\cdot\sup_{m=0,\dots,2^{j}-1}|\lambda_{j,m}|&=\frac{1}{\sqrt{j}}\cdot \|\lambda_{j,\cdot}\|_\infty\le \frac{1}{\sqrt{j}}\cdot \|\lambda_{j,\cdot}\|_{j}=
 \frac{2}{\sqrt{j}}\,\Bigl(\sum_{m=0}^{2^j-1}2^{-j}|\lambda_{j,m}|^j\Bigr)^{1/j}\\
 &\le 2\sup_{p\ge 1}\frac{1}{\sqrt{p}}\,\Bigl(\sum_{m=0}^{2^j-1}2^{-j}|\lambda_{j,m}|^p\Bigr)^{1/p}.
 \end{align*}
Therefore, $B^{1/2}_{{\Phi_2},\infty}(I) \hookrightarrow B^{1/2,1/2}_{\infty,\infty}(I)$ and the result of Ciesielski is an improvement over the result of P. L\'evy, providing a strictly smaller space, which contains almost all Wiener paths.

Finally, let us note that $B^{1/2,1/2}_{\infty,\infty}(I)$ and $B^{1/2}_{p,\infty}(I)$ with $1\le p<\infty$ are incomparable. This can again be easily seen by looking at the sequence space characterization \eqref{eq:char_1} and \eqref{eq:char_2}.
First, the special sequence $\lambda^{(1)}_{j,m}=\sqrt{j}$ for all $j\in\N$ and all $m=0,\dotsc,2^j-1$ belongs to $B^{1/2,1/2}_{\infty,\infty}(I)$ and not to $B^{1/2}_{p,\infty}(I)$ for any $1\leq p<\infty$. Second, the sequence
\[
\lambda^{(2)}_{j,m}=\begin{cases} j, \quad &\text{for}\ j\in\N\ \text{and}\ m=0,\\
0, \quad &\text{for}\ j\in\N\ \text{and}\ m\ge 1
\end{cases}
\]
belongs to $B^{1/2}_{p,\infty}(I)$ for all $1\leq p<\infty$ but not to $B^{1/2,1/2}_{\infty,\infty}(I)$.
\end{rem}

\subsection{An alternative proof for the Besov-Orlicz space $B^{1/2}_{\Phi_2,\infty}(I)$}\label{sub:2.4}
We use the characterization given in Theorem \ref{thm:Orlicz} to re-prove the result of Ciesielski \cite{Ciesiel93}, i.e., to show that the Wiener paths lie 
almost surely in the Besov-Orlicz space $B^{1/2}_{\Phi_2,\infty}(I)$. Comparing \eqref{eq:Levy01} with \eqref{eq:decomp_Faber1} we observe, that it is enough to show that
\[
\left\|\xi|b_{\Phi_2,\infty}^+(I)\right\|=\sup_{j\in\N_0}\left\|\sum_{m=0}^{2^j-1}\xi_{j,m}\chi_{j,m}(\cdot)\right\|_{\Phi_2}<\infty
\]
almost surely.  Here,  $\xi=\{\xi_{j,m},j\in\N_0,0\le m\le 2^j-1\}$ is again a sequence of i.i.d. standard Gaussian variables and
$\chi_{j,m}$ denotes the characteristic function of the interval $I_{j,m}$.

Therefore, we put for every integer $j\ge 0$
\begin{equation}\label{eq:Def_fj}
f_j(t)=\sum_{m=0}^{2^j-1}\xi_{j,m}\chi_{j,m}(t),\quad 0<t< 1,
\end{equation}
and observe that its non-increasing rearrangement is given by 
\[
f_j^{\ast}(t)=\sum_{m=0}^{2^j-1}\left(\xi_{j}\right)^{\ast}_{m+1}\chi_{j,m}(t),\quad 0<t< 1,
\]
where $\left(\left(\xi_{j}\right)^{\ast}_m\right)_{m=1}^{2^j}$ is the non-increasing rearrangement of the sequence  $\xi_j=(|\xi_{j,m}|)_{m=0}^{2^j-1}$.

This allows us to calculate for $j\in \N$ fixed 
\begin{align}\notag
\sup_{0<t<1}\frac{f_j^*(t)}{\sqrt{\log(1/t)+1}}&=\sup_{m=1,\dots,2^j}\frac{f_j^*(m2^{-j})}{\sqrt{\log(2^{j}/m)+1}}\\
\label{eq:flog_bound}&=\sup_{k=0,\dots,j-1}\sup_{2^k\le m\le 2^{k+1}}\frac{f_j^*(m2^{-j})}{\sqrt{\log(2^{j}/m)+1}}\\
&\notag\le \sup_{k=0,\dots,j-1}\frac{f_j^*(2^{k-j})}{\sqrt{\log(2^{j-k-1})+1}}\le
c\sup_{0\le k<j}\frac{f_j^*(2^{k-j})}{\sqrt{j-k}}\\
\notag&=c\sup_{0\le k<j}\frac{\left(\xi_{j}\right)^{\ast}_{2^k}}{\sqrt{j-k}}. 
\end{align}
Next, we denote by $A_{j,k}^K$ the event that
$
\left(\xi_{j}\right)^{\ast}_{2^k} \ge K\sqrt{j-k}
$
and use Lemma \ref{lem:tail_bounds} to estimate $\P(A_{j,k}^K)$.
We conclude, that for every $K_0\in\N$ large enough such that $16 e^{-K_0^2/2}<1$ it holds
\begin{align*}
\P\biggl( & \sup_{0<t<1}\sup_{j\in\N_0}\frac{f_j^*(t)}{\sqrt{\log(1/t)+1}}=\infty\biggr)
=\P\Bigl(\bigcap_{K=1}^\infty\bigcup_{0\leq k<j<\infty} A_{j,k}^K\Bigr)
\le \P\Bigl(\bigcup_{0\leq k<j<\infty} A_{j,k}^{K_0}\Bigr)\\
&\le \sum_{0\leq k<j<\infty} \P\bigl(A_{j,k}^{K_0}\bigr)\le 
 \sum_{0\leq k<j<\infty}\Bigl(2e^{-K_0^2/2}\Bigr)^{(j-k)2^k}\cdot {e}^{2^k}\\
&= \sum_{k=0}^{\infty}\sum_{l=1}^{\infty}\Bigl(2e^{-K_0^2/2}\Bigr)^{l2^k}\cdot {e}^{2^k}
= \sum_{k=0}^{\infty}{e}^{2^k}\sum_{l=1}^{\infty}\left(\Bigl(2e^{-K_0^2/2}\Bigr)^{2^k}\right)^{l} \\
&\le c\sum_{k=0}^\infty e^{2^k}\Bigl(2e^{-K_0^2/2}\Bigr)^{2^k}\le
c\sum_{k=0}^\infty \Bigl(8e^{-K_0^2/2}\Bigr)^{k+1}
\le 2c\cdot 8e^{-K_0^2/2}.
\end{align*}

The last expression tends to zero if $K_0\to\infty$, which yields the desired result. 

\subsection{New function spaces}\label{sec:2.5}

The approach to regularity of Wiener paths presented in Sections \ref{sub:2.3} and \ref{sub:2.4} follows actually a rather
straightforward pattern. If a certain function space under consideration allows for an equivalent characterization
in terms of the Faber system, then this might be combined directly with Theorem \ref{thm:LevyDecomp}. The proof that almost all
Wiener paths lie in this space then reduces to a statement about independent standard Gaussian variables.

In this section we show, that this approach can
 be further developed to introduce even smaller spaces than $B^{1/2}_{\Phi_2,\infty}(I)$,
where the Wiener paths lie in almost surely.

\subsubsection{Spaces of Besov type - discrete averages of differences}

By Definition \ref{def:orlicz_besov} and Theorem \ref{thm:Orlicz}, a continuous function $f\in C(I)$ lies in the Besov-Orlicz space
$B_{\Phi_2,\infty}^{1/2}(I)$ if, and only if, its decomposition into the Faber system \eqref{eq:decomp_Faber1} satisfies
\begin{equation*}
\sup_{j\in\N_0} \sup_{0<t<1}\frac{f^*_j(t)}{\sqrt{\log(1/t)+1}}<\infty,\quad
\end{equation*}
where
\begin{equation}\label{eq:new_besov01}
f_j(t)=\sum_{m=0}^{2^j-1}\lambda_{j,m}\chi_{j,m}(t),\quad 0<t< 1.
\end{equation}
This condition quantifies very precisely the possible size of the components $f_j$, but (due to the use of the rearrangements $f_j^*$)
it fails to describe the distribution of large values of $f_j$ on $[0,1]$.
For example, it does not exclude the possibility that the large values of the components $f_j$ appear close to each other.
But this is actually unlikely for the Wiener process because in that case, the $\lambda_{j,m}$'s get replaced by independent Gaussian variables.

Therefore, we introduce new function spaces, where we measure the size of the averages
\[
\frac{1}{t-s}\int_s^t f_j(u)du\quad \text{or}\quad \frac{1}{t-s}\int_s^t |f_j(u)|du
\]
and we expect them to be much smaller (in $L_\infty$ or $L_{\Phi_2}$-norm) than $f_j$ itself.
To describe this idea mathematically, we define for an integrable function $g$ on $I$ the averaging operators
\begin{equation}\label{eq:AV1}
(A_kg)(x)=\sum_{l=0}^{2^k-1}2^k\int_{l2^{-k}}^{(l+1)2^{-k}}g(t)dt\cdot\chi_{k,l}(x)\quad\text{and}\quad (\widetilde A_kg)(x)=(A_k(|g|))(x),\quad x\in I.
\end{equation}

By what we outlined so far, we expect   the $k$th dyadic average $A_k$ of the $j$th dyadic level of  L\'evy's decomposition
\eqref{eq:Levy01}   to be much smaller in the $L_{\Phi_2}$-norm than the $j$th dyadic level itself. This paves the way to the following definition.
\begin{dfn}
Let $\varepsilon>0$. Then the space $A(\varepsilon)$ is the collection of all $f\in C(I)$, which satisfy the following condition.
If \eqref{eq:decomp_Faber1} is the decomposition of $f$ into the Faber system and $f_j$, $j\in \N_0$, is defined by \eqref{eq:new_besov01}, then
\[
\|f\|_{A(\varepsilon)}:=\sup_{j\in\N_0}\sup_{0\le k\le j} \sup_{0<t<1} \frac{2^{(j-k)/2}}{(j-k+1)^\varepsilon}\cdot\frac{(A_kf_j)^*(t)}{\sqrt{\log(1/t)+1}}
\approx \sup_{j\in\N_0}\sup_{0\le k\le j} \frac{2^{(j-k)/2}}{(j-k+1)^\varepsilon}\cdot \|A_kf_j\|_{\Phi_2}<\infty.
\]
\end{dfn}

With this definition we are now able to prove the following.

\begin{thm}\label{thm:Aepsilon}
Let $\varepsilon>0$. Then $\|W_{\cdot}\|_{A(\varepsilon)}<\infty$ almost surely.
\end{thm}
\begin{proof}
To estimate the norm of the Brownian paths in $A(\varepsilon)$, we use the decomposition of $W_t$ into the Faber system \eqref{eq:Levy01}.
We therefore replace \eqref{eq:new_besov01} by
\begin{equation*}
f_j(t)=\sum_{m=0}^{2^j-1}\xi_{j,m}\chi_{j,m}(t),\quad 0<t< 1,
\end{equation*}
where $\xi_j=(\xi_{j,m})_{m=0}^{2^j-1}$ is again a vector of independent standard Gaussian variables.

Let $0\le k\le j$ and $x\in I$. Then
\begin{align*}
    2^{(j-k)/2}A_kf_j(x)&=2^{(j-k)/2}\sum_{l=0}^{2^k-1}2^k\int_{l\cdot 2^{-k}}^{(l+1)2^{-k}}\Biggl(\sum_{m=0}^{2^j-1}\xi_{j,m}\chi_{j,m}(t)\Biggr)dt\cdot \chi_{k,l}(x)\\
    &=\sum_{l=0}^{2^k-1}2^{(j+k)/2}\Biggl(\sum_{m=0}^{2^j-1}\xi_{j,m}|I_{k,l}\cap I_{j,m}|\Biggr)\chi_{k,l}(x)\\
    &=\sum_{l=0}^{2^k-1}2^{(k-j)/2}\Biggl(\sum_{m=l\cdot 2^{j-k}}^{(l+1)\cdot 2^{j-k}-1}\xi_{j,m}\Biggr)\chi_{k,l}(x).
\end{align*}
By the 2-stability of Gaussian variables, cf. Lemma \ref{2-stability}, it follows that $2^{(j-k)/2}A_kf_j$ is equidistributed with $f_k$. Therefore,
\begin{align*}
\P(\|W_{\cdot}\|_{A(\varepsilon)}\ge K_0)&\le \sum_{0\le k\le j} \P\left(\frac{2^{(j-k)/2}}{(j-k+1)^\varepsilon}\cdot\sup_{0<t<1}\frac{(A_kf_j)^*(t)}{\sqrt{\log(1/t)+1}}\ge K_0\right)\\
&= \sum_{0\le k\le j} \P\left(\sup_{0<t<1}\frac{(f_k)^*(t)}{\sqrt{\log(1/t)+1}}\ge K_0\cdot (j-k+1)^\varepsilon\right)=I_1+I_2,
\end{align*}
where $I_1$ collects the terms with $0=k\le j$ and $I_2$ includes the terms with $1\le k\le j$.

The estimate of $I_1$ is rather straightforward
\begin{align*}
I_1\le \sum_{j=0}^\infty \P\Bigl(f_0^*(1)\ge K_0\cdot (j+1)^\varepsilon\Bigr)
\le \sum_{j=0}^\infty \exp\Bigl(-\frac{K_0^2(j+1)^{2\varepsilon}}{2}\Bigr)
\end{align*}
and this expression tends to zero if $K_0$ grows to infinity.

Using \eqref{eq:flog_bound} and Lemma \ref{lem:tail_bounds}, we may estimate $I_2$ as follows
\begin{align*}
I_2&\le \sum_{0\le m< k\le j} \P\left(\frac{(\xi_k)^*_{2^m}}{\sqrt{k-m}}\ge c_1 K_0\cdot (j-k+1)^\varepsilon\right)\\
&\le \sum_{0\le m< k\le j}c\,e^{2^m}\Bigl(2\exp(-c_1^2K_0^2(j-k+1)^{2\varepsilon})\Bigr)^{(k-m)2^m}\\
&= \sum_{0\le m< k}c\,e^{2^m}2^{(k-m)2^m}\sum_{l=1}^\infty\Bigl(\exp(-c_1^2K_0^2(k-m)2^m)\Bigr)^{l^{2\varepsilon}}.
\end{align*}
We assume that $K_0$ is large enough to ensure $8\exp(-c_1^2K_0^2)<1/2$ and obtain
\begin{align*}
I_2&\le c_\varepsilon\sum_{0\le m<k} e^{2^m}2^{(k-m)2^m}\exp(-c_1^2K_0^2(k-m)2^m)\\
&=c_\varepsilon\sum_{m=0}^\infty e^{2^m}\sum_{\nu=1}^\infty \Bigl(2^{2^m}\exp(-c_1^2K_0^22^m)\Bigr)^{\nu}\le 2c_\varepsilon \sum_{m=0}^\infty 8^{2^m} \exp(-c_1^2K_0^22^m)\\
&\le 32 c_\varepsilon \exp(-c_1^2K_0^2),
\end{align*}
which tends again to zero if $K_0$ grows to infinity.
\end{proof}

A similar result can be obtained if we use the absolute averaging operators $\widetilde A_k$ instead of $A_k$.
The decay of the averages $\widetilde A_k(f_j)$ will now be described by the Orlicz spaces $L_{\Phi_{2,A}}$,
which are defined in \eqref{eq:Orlicz_new1}, also cf.  Theorem \ref{thm:OrliczA}.

\begin{dfn}\label{dfn:NewSpaceB1}
The space $\widetilde A$ is the collection of all $f\in C(I)$, which satisfy the following condition.
If \eqref{eq:decomp_Faber1} is the decomposition of $f$ into the Faber system and $f_j$, $j\in \N_0$, is defined by \eqref{eq:new_besov01}, then
\[
\|f\|_{\widetilde A}:=\sup_{j\in\N_0}\sup_{0\le k\le j} \sup_{0<t<1} \frac{(\widetilde A_kf_j)^*(t)}{\sqrt{2^{k-j}\log(1/t)+1}}
\approx \sup_{j\in\N_0}\sup_{0\le k\le j} \|{\widetilde A}_kf_j\|_{2,2^{k-j}}<\infty.
\]
\end{dfn}

With this we can now state and  prove the following statement. 

\begin{thm}\label{thm:13}
It holds that $\|W_{\cdot}\|_{\widetilde A}<\infty$ almost surely.
\end{thm}
\begin{proof}
By its construction,
\[
\widetilde A_kf_j(x)=\sum_{m=0}^{2^k-1}\nu_{k,m}\chi_{k,m}(x),\quad x\in[0,1],
\]
where $\nu_k=(\nu_{k,0},\dots,\nu_{k,2^k-1})$ is a vector of independent variables from ${\mathcal G}_{2^{j-k}}$, see Definition \ref{dfn:G_N}. Similarly to \eqref{eq:flog_bound}, we obtain
\begin{align*}
\sup_{0<t<1} &\frac{(\widetilde A_kf_j)^*(t)}{\sqrt{2^{k-j}\log(1/t)+1}}=\sup_{m=1,\dots,2^k} \frac{(\widetilde A_kf_j)^*(m2^{-k})}{\sqrt{2^{k-j}\log(2^k/m)+1}}\\
&\le c\,\sup_{z=0,\dots,k-1}\frac{(\widetilde A_kf_j)^*(2^{z-k})}{\sqrt{2^{k-j}(k-z-1)+1}}
=c\,\sup_{z=0,\dots,k-1}\frac{(\nu_k^*)_{2^{z}}}{\sqrt{2^{k-j}(k-z-1)+1}}.
\end{align*}

Then, for $K_0$ large enough using the estimate above and \eqref{eq:GN1}, we obtain
\begin{align*}
\P(\|W_{\cdot}\|_{\widetilde A}\ge K_0)&\le \sum_{0\le k\le j}\P\biggl(\sup_{0<t<1}\frac{(\widetilde A_kf_j)^*(t)}{\sqrt{2^{k-j}\log(1/t)+1}}\ge K_0\biggr)\\
&\le \sum_{0\le z<k\le j}\P\biggl(\frac{(\nu_k^*)_{2^{z}}}{\sqrt{2^{k-j}(k-z-1)+1}}\ge cK_0\biggr)\\
&\le \sum_{0\le z<k\le j} \biggl[e\cdot 2^{k-z}\exp(-(2^{k-j}(k-z-1)+1)c^2K_0^22^{j-k}/4)\biggr]^{2^z}\\
&=\sum_{z=0}^\infty e^{2^z}\sum_{k=z+1}^\infty 2^{(k-z)2^z}\exp\Bigl[-(k-z-1)c^2K_0^2/4\cdot 2^z\Bigr]\sum_{j=k}^\infty\exp\Bigl[-2^{j-k}c^2K_0^2/4\cdot 2^z\Bigr]\\
&\le 2\sum_{z=0}^\infty e^{2^z}\sum_{k=z+1}^\infty 2^{(k-z)2^z}\exp\Bigl[-(k-z-1)c^2K_0^2/4\cdot 2^z\Bigr]\exp\Bigl[-c^2K_0^2/4\cdot 2^z\Bigr]\\
&= 2\sum_{z=0}^\infty e^{2^z}\sum_{k=z+1}^\infty \biggl(2^{2^z}\exp\Bigl[-c^2K_0^2/4\cdot 2^z\Bigr]\biggr)^{k-z}\\
&\le 4\sum_{z=0}^\infty e^{2^z}\cdot 2^{2^z}\exp\Bigl[-c^2K_0^2/4\cdot 2^z\Bigr]\le 4\sum_{z=0}^\infty \Bigl(8e^{-c^2K_0^2/4}\Bigr)^{2^z}
\le 8\cdot 8e^{-c^2K_0^2/4}.
\end{align*}
As the last expression tends to zero if $K_0\to\infty$, this finishes the proof.
\end{proof}
\begin{rem}
Setting $j=k$ in the definition of $A(\varepsilon)$ or $\widetilde A$ and observing that $A_jf_j=f_j$ and $\widetilde A_j f_j=|f_j|$, we conclude that $A(\varepsilon)$ and $\widetilde A$ are both subsets of the space $B^{1/2}_{\Phi_2,\infty}(I)$ considered by  Ciesielski.\\
On the other hand, if we put
\[
\lambda_{j,m}=\sqrt{\log\Bigl(\frac{2^j}{m+1}\Bigr)+1},\quad j\in\N_0,\quad m=0,\dots,2^j-1, 
\]
and define $f_j$ by \eqref{eq:new_besov01}, then it follows from  $(A_kf_j)(t)=(\widetilde{A}_kf_j)(t)\geq f_k(t)$ that $A(\varepsilon)$ and $\widetilde A$ are proper subsets of $B^{1/2}_{\Phi_2,\infty}(I)$.
\end{rem}

\subsubsection{Spaces of Besov type - continuous averages of differences}

The function space $\tilde A$ introduced in Definition \ref{dfn:NewSpaceB1} is somehow difficult to handle. In order to decide whether a continuous function $f$
belongs to $\tilde A$, we first have   to construct its Faber decomposition \eqref{eq:decomp_Faber1} and the sequence $\{f_j\}_{j=0}^\infty$, cf. \eqref{eq:new_besov01}. Afterwards, we need to apply the
averaging operators $\widetilde A_k$ of \eqref{eq:AV1} and, finally, we have to measure the size of $\widetilde A_k f_j$ in the corresponding Orlicz space $L_{\Phi_2}(I)$.

Therefore, we investigate if $\widetilde A$ could be possibly replaced by a space which is defined more directly, without the detour through the Faber system decomposition.
For this, we first note  that by Theorem \ref{thm:OrliczA}
\[
\|f\|_{\widetilde A}\asymp\sup_{0\le k\le j} \|\widetilde A_kf_j\|_{2,2^{k-j}},
\]
where $\|\cdot\|_{2,2^{k-j}}$ is the Orlicz norm introduced in \eqref{eq:OrliczA}.
To avoid the use of $f_j$ and $\widetilde A_k$, we observe that
\[
f_{j}(t)=\sum_{m=0}^{2^j-1}\lambda_{j,m}\chi_{j,m}(t)\quad \text{and}\quad \lambda_{j,m}=-2^{j/2}(\Delta^2_{2^{-j-1}}f)(m\cdot 2^{-j})
\]
gives for $0\le k\le j$
\begin{align*}
\widetilde A_k f_j(x) &=\sum_{l=0}^{2^k-1}2^k\int_{I_{k,l}}|f_j(t)|dt\cdot \chi_{k,l}(x)\\
&=\sum_{l=0}^{2^k-1}\chi_{k,l}(x)\sum_{m=0}^{2^j-1}2^k\cdot|\lambda_{j,m}|\cdot|I_{k,l}\cap I_{j,m}|\\
&=2^{j/2}\sum_{l=0}^{2^k-1}\chi_{k,l}(x)\cdot \frac{1}{2^{j-k}}\sum_{m=l\cdot 2^{j-k}}^{(l+1)2^{j-k}-1}|(\Delta^2_{2^{-j-1}}f)(m\cdot 2^{-j})|.
\end{align*}
We now replace the discrete averages of second order differences by the continuous averages. Before we come to that, we need to complement \eqref{differences} by the differences restricted to $I$ and set 
\begin{equation}\label{eq:seconddiffI}
\Delta^2_hf(x)=\begin{cases} f(x+2h)-2f(x+h)+f(x)\quad &\text{if}\  \{x,x+h,x+2h\}\subset I,\\
0\quad &\text{otherwise}.
\end{cases}
\end{equation}

This paves the way for the following defintion. 

\begin{dfn}\label{def:DD}
Let $f\in C(I)$.
\begin{enumerate}
\item Then we define for every $0\le k\le j$ and every $x\in I$
\begin{align*}
D^2_{j,k}f(x)&:=\sum_{l=0}^{2^k-1}\chi_{k,l}(x)\frac{1}{2^{-k}}\int_{l\cdot 2^{-k}}^{(l+1)2^{-k}}|\Delta^2_{2^{-j-1}}f(t)|dt,
\intertext{and}
{\mathfrak D}^2_{j,k}f(x)&:=\frac{1}{2^{-k}}\int_{x-2^{-k-1}}^{x+2^{-k-1}}|\Delta^2_{2^{-j-1}}f(t)|dt.
\end{align*}
\item
We define
\begin{equation}\label{eq:NewNorm01}
\|f\|_{D}=\sup_{j\in\N_0}\sup_{0\le k\le j}\left\|2^{j/2}D^2_{j,k}f\right\|_{2,2^{k-j}}
\end{equation}
and
\begin{equation}\label{eq:NewNorm2}
\|f\|_{{\mathfrak D}}=\sup_{j\in\N_0}\sup_{0\le k\le j}\left\|2^{j/2}{\mathfrak D}^2_{j,k}f\right\|_{2,2^{k-j}}.
\end{equation}
\end{enumerate}
\end{dfn}

Now we are in a position to state and prove the following.

\begin{thm} \label{thm:15}
\begin{enumerate}
\item[(i)] Let $f\in C(I)$. Then $\|f\|_{D}\approx \|f\|_{\mathfrak{D}}$.
\item[(ii)] $\|W_\cdot\|_{D}<\infty$ almost surely.
\item[(iii)] $\|W_\cdot\|_{\mathfrak{D}}<\infty$ almost surely.
\end{enumerate}
\end{thm}
\begin{proof}
Step 1. Let $x\in I_{k,l}$ for some $l\in\{0,\dots,2^k-1\}$.
Then $I_{k,l}\subset (x-2^{-k},x+2^{-k})$, which gives $D^2_{j,k}f(x)\le 2{\mathfrak D}^2_{j,k-1}f(x)$.
This allows to estimate the terms with $0<k\le j$ in \eqref{eq:NewNorm01} from above by $\|f\|_{\mathfrak D}$.

To estimate also the terms with $0=k\le j$,
we put  $g(t)=|\Delta^2_{2^{-j-1}}f(t)|$ and calculate
\begin{align*}
\left\|D^2_{j,0}f\right\|_{2,2^{-j}}
&=\left\|\int_0^1 g(t)dt\cdot \chi_{[0,1]}\right\|_{2,2^{-j}}\\
&=
\left\|\int_0^{1/2} g(t)dt\cdot \chi_{[0,1]}\right\|_{2,2^{-j}}+ 
\left\|\int_{1/2}^1 g(t)dt\cdot \chi_{[0,1]}\right\|_{2,2^{-j}}\\
&\le C\left\{\left\|\int_0^{1/2} g(t)dt\cdot \chi_{[0,1/2]}\right\|_{2,2^{-j}}+ 
\left\|\int_{1/2}^1 g(t)dt\cdot \chi_{[1/2,1]}\right\|_{2,2^{-j}}\right\}\\
&\le C\left\{\left\|{\mathfrak D}^2_{j,0}f(x)\cdot \chi_{[0,1/2]}(x)\right\|_{2,2^{-j}}+ 
\left\|{\mathfrak D}^2_{j,0}f(x)\cdot \chi_{[1/2,1]}(x)\right\|_{2,2^{-j}}\right\}\\
&\le 2C \left\|{\mathfrak D}^2_{j,0}f\right\|_{2,2^{-j}},
\end{align*}
which gives that $\|f\|_{D}\lesssim \|f\|_{\mathfrak D}$.

Step 2. Fix $0\le k\le j$ and set again $g(t)=|\Delta^2_{2^{-j-1}}f(t)|$. If $x\in I_{k,l}$, then $(x-2^{-k-1},x+2^{-k-1})\subset I_{k,l-1}\cup I_{k,l}\cup I_{k,l+1}$. Therefore,
\begin{align*}
{\mathfrak D}_{j,k}^2 f(x)
=\sum_{l=0}^{2^k-1}{\mathfrak D}_{j,k}^2 f(x)\chi_{k,l}(x)
\le \sum_{l=0}^{2^k-1}\frac{1}{2^{-k}}\left\{\int_{I_{k,l-1}}|g(t)|dt+\int_{I_{k,l}}|g(t)|dt+\int_{I_{k,l+1}}|g(t)|dt\right\}\chi_{k,l}(x)
\end{align*}
We now apply the shift-invariance of the space $L_{{2,2^{k-j}}}$ and obtain that
$\|f\|_{\mathfrak D}\lesssim \|f\|_{D}$.

Step 3.  It is clear that (iii) follows from (i) and (ii). Hence, it is enough to prove (ii),
i.e., we show $\|W_\cdot\|_D<\infty$ almost surely. Although the proof resembles the proof of Theorem \ref{thm:13}, we will omit the use of the Faber system in this case. Moreover, in order to avoid technicalities, we first make the following observation. When we use the values of $(W_t)_{t\ge 0}$ also for $t>1$, we can assume that
\[
\Delta^2_hW(x)=W(x+2h)-2W(x+h)+W(x)
\]
for every $x,h>0$. This means that we do not make use of the restriction to $I$ as it appeared in \eqref{eq:seconddiffI}, which can make $\|W_\cdot\|_D$ only larger.

Furthermore, we distinguish again between $k=0$ and $k\ge 1$. If $k=0$, then
\begin{align*}
D^2_{j,0}W(x)=\chi_{I}(x)\cdot \int_0^1 |(\Delta_{2^{-j-1}}^2W)(t)|dt
\end{align*}
and, using Theorem \ref{thm:OrliczA}, we see that 
\begin{align}
\notag\sup_{j\ge 0} 2^{j/2}\|D^2_{j,0}W\|_{2,2^{-j}}
&=\sup_{j\ge 0}2^{j/2}\int_0^1 |(\Delta_{2^{-j-1}}^2W)(t)|dt 
\cdot \|\chi_{I}(x)\|_{2,2^{-j}}\\
\label{eq:Dspaces1}&\approx \sup_{j\ge 0} 2^{j/2}\int_0^1 |(\Delta_{2^{-j-1}}^2W)(t)|dt\\
\notag&=\sup_{j\ge 0}2^{j/2}\sum_{m=0}^{2^{j+1}-1} \int_{m\cdot 2^{-j-1}}^{(m+1)2^{-j-1}}|(\Delta_{2^{-j-1}}^2W)(t)|dt.
\end{align}

If $m\cdot 2^{-j-1}\le t\le (m+1)2^{-j-1}$, then we estimate
\begin{align}
\notag |(\Delta_{2^{-j-1}}^2W)(t)|
&=|W(t+2^{-j})-2W(t+2^{-j-1})+W(t)|\\
\label{eq:Dspaces2}&\le |W(t+2^{-j})-W((m+2)2^{-j-1})|
+|W((m+2)2^{-j-1})-W(t+2^{-j-1})|\\\notag &\quad+|W((m+1)2^{-j-1})-W(t+2^{-j-1})| + |W(t)-W((m+1)2^{-j-1})|.
\end{align}
If we plug this estimate into \eqref{eq:Dspaces1}, we obtain four (very similar) terms. We only estimate the first term, since  the others can be handled in the same manner. If we set 
\[
\alpha^j_m=\frac{1}{2^{-j-1}}\int_{m\cdot 2^{-j-1}}^{(m+1)2^{-j-1}}|W(t)-W(m\cdot2^{-j-1})|dt,
\]
we can see that 
\begin{align*}
\sup_{j\ge 0}\, & 2^{j/2}\sum_{m=0}^{2^{j+1}-1} \int_{m\cdot 2^{-j-1}}^{(m+1)2^{-j-1}}|W(t+2^{-j})-W((m+2)2^{-j-1})|dt=\sup_{j\ge0}2^{-j/2-1}\sum_{m=0}^{2^{j+1}-1} \alpha^j_{m+2}.
\end{align*}
By their definition, $\{\alpha^j_m\}_{m\ge 0}$ are independent random variables, all equidistributed with
\[
\frac{1}{2^{-j-1}}\int_0^{2^{-j-1}}|W(s)|ds
\]
and  therefore have  the same distribution as
$2^{-(j+1)/2}{\mathcal W}$, where ${\mathcal W}=\int_0^1|W(s)|ds$ is the integrated absolute Wiener process. For the convenience of the reader we recall a few of its basic properties in Section \ref{sec:IW}. In particular, Lemma \ref{lem:IW} (ii) allows us to conclude for every $K\ge 1$ that 
\begin{align*}
\P\left(\sup_{j\ge 0}2^{-j/2-1}\sum_{m=0}^{2^{j+1}-1} \alpha^j_{m+2}=\infty\right)
\le \sum_{j=0}^\infty \P\left(\frac{1}{2^{j+1}}\sum_{m=0}^{2^{j+1}-1}2^{\frac{j+1}{2}}\alpha^j_m>K\right)
\le \sum_{j=0}^\infty \exp(1-c 2^{j+1}K^2),
\end{align*}
which goes to zero if $K\to\infty.$

Step 4. Next, we estimate the terms with $0<k\le j$. The argument is quite similar to the previous step, which allows us to leave out some technical details. First for every $k\ge 1$ we rewrite 
\begin{align}
\notag D^2_{j,k}W(x)&=\sum_{l=0}^{2^k-1}\chi_{k,l}(x)\cdot\frac{1}{2^{-k}}\int_{l\cdot 2^{-k}}^{(l+1)2^{-k}}|\Delta^2_{2^{-j-1}}W(t)|dt\\
\label{eq:new1}&=\sum_{l=0}^{2^k-1}\chi_{k,l}(x)\cdot\frac{1}{2^{-k}}\sum_{m=0}^{2^{j-k+1}-1}\int_{l\cdot 2^{-k}+m\cdot 2^{-j-1}}^{l\cdot2^{-k}+(m+1)2^{-j-1}}|\Delta^2_{2^{-j-1}}W(t)|dt.
\end{align}
If $0\le t-l\cdot 2^{-k}-m\cdot 2^{-j-1}\le 2^{-j-1}$, then we obtain \eqref{eq:Dspaces2}
with $m\cdot 2^{-j-1}$ replaced by
$\tau:=l\cdot 2^{-k}+m\cdot 2^{-j-1}$.
We insert this into \eqref{eq:new1} and derive again   an estimate of $D^2_{j,k}W(x)$ invoking a sum of four terms, which we denote by $D^{2,i}_{j,k}W(x)$ with $i\in\{1,2,3,4\}$. Again, we estimate only one of these terms (say, the one with $i=4$), since the others are very similar to deal with. We put 
\[
\alpha^{j,k}_{l,m}=\frac{1}{2^{-j-1}}\int_{l\cdot 2^{-k}+m\cdot 2^{-j-1}}^{l\cdot2^{-k}+(m+1)2^{-j-1}}|W(t)-W(l\cdot 2^{-k}+m\cdot 2^{-j-1}+2^{-j-1})|dt
\]
and obtain
\[
D^{2,4}_{j,k}W(x)=\sum_{l=0}^{2^k-1}\chi_{k,l}(x)\cdot\frac{1}{2^{j-k+1}}\sum_{m=0}^{2^{j-k+1}-1}\alpha^{j,k}_{l,m}.
\]
As $V_t=W_1-W_{1-t}$ with $0\le t\le 1$ is equidistributed with $W_t$, we observe that 
\[
\{\alpha_{j,k}^{l,m}:l=0,\dots,2^k-1, m=0,\dots,2^{j-k+1}-1\}
\]
are independent random variables distributed like
$2^{-(j+1)/2}{\mathcal W}$. 
Hence, if we set 
\[
B^{j,k}_l=\frac{1}{2^{j-k+1}}\sum_{m=0}^{2^{j-k+1}-1}2^{(j+1)/2}\alpha^{j,k}_{l,m},
\]
then $\{B^{j,k}_l\}_{l=0}^{2^k-1}$ are independent and each of them is distributed as the average of $2^{j-k+1}$ variables from ${\mathcal W}.$
Finally, we observe that $\|2^{j/2}D^{2,4}_{j,k}W\|_{2,2^{k-j}}$ has the same distribution as
\[
\frac{1}{\sqrt{2}}\left\|\sum_{l=0}^{2^k-1}\chi_{k,l}(x)B^{j,k}_l\right\|_{2,2^{k-j}}.
\]
The proof is then finished in the same manner as the proof of Theorem \ref{thm:13} by using \eqref{eq:IntW04} instead of \eqref{eq:GN1}. 
\end{proof}

We close this section with the following  remark and two open problems.
\begin{rem}
A natural question at this place is, if the method and the results presented so far could also be applied to other processes. The first natural candidate is the \emph{fractional Brownian motion}, which is a Gaussian process $B_H=(B_H(t))_{t\ge 0}$ with $B_H(0)=0$, $\E B_H(t)=0$ for all $t\ge 0$ and
\[
\E[B_H(t)B_H(s)]=\frac{1}{2}(|t|^{2H}+|s|^{2H}-|t-s|^{2H})\quad \text{for all}\ s,t\ge 0.
\]
Here, $H\in(0,1)$ is the so-called Hurst index. For $H=1/2$, one recovers the standard Brownian motion with independent increments, but the increments are no longer independent if $H\not=1/2$.

One can again recover the decomposition of the paths of $B_H$ into the Faber system, but (due to loss of independence of the increments) the random coefficients are no longer independent, see \cite[Lemma IV.2]{CKR93}. Although this obstacle can be overcome, the proofs are technically much more involved. That is why we decided to concentrate on the standard Brownian motion in this paper. 
\end{rem}

\begin{OP} All the regularity results obtained so far used Besov spaces and their numerous variants. The other well-known scale of Fourier-analytic function spaces, the so-called \emph{Triebel-Lizorkin spaces}, did not play any important role up to now. For these space a characterization by the Faber system is also available and actually quite similar to Theorem \ref{thm:faber_s}. The main difference to Besov spaces is that one first applies some sequence space norm to $(f_j(t))_{j\ge 0}$, cf. \eqref{eq:Def_fj}, and only afterwards some function space norm. The analysis of the regularity of Brownian paths in the frame of Triebel-Lizorkin spaces could be based on the observation, that the functions $f_j(t)$ are extremely unlikely to be large for the same value of $t$.
\end{OP}

\begin{OP}
It is very well known, that Besov (and Triebel-Lizorkin) spaces can be characterized by differences in several different ways, including also the so-called \emph{ball means of differences}, cf. \cite{TrFS1}. Note however, that the usual ball means of differences differ essentially from $D^2_{j,k}f(x)$ and ${\mathfrak D}^2_{j,k}f(x)$ introduced in Definition \ref{def:DD}. It would be of some interest to know, if one could build new scales of function spaces in the spirit of \eqref{eq:NewNorm01} and \eqref{eq:NewNorm2}
and investigate their relation  with the  already known function spaces.
\end{OP}

\section{Path regularity of Brownian sheets}\label{sec:3}

The aim of this section is to show that the methods presented in Section \ref{sec:d1} for the univariate Wiener process
can be quite easily generalized to the multivariate setting. First, let us introduce the multivariate
analogue of the Wiener process, the so-called Brownian sheet.
\begin{dfn}\label{dfn:BrownianSheet} A continuous Gaussian process $B=(B(t))_{t\in\R_+^d}$ is called Brownian sheet, if $\E B(t)=0$ for every $t=(t_1,\dots,t_d)\in\R_+^d$ and
\begin{equation}\label{Brownian-sheet-1}
{\rm Cov}(B(s),B(t))=\prod_{i=1}^d \min(s_i,t_i)\quad \text{for all}\quad s,t\in\R_+^d.
\end{equation}
\end{dfn}

The study of the Brownian sheet goes back to the 1950's. Since then, many of its properties (including the properties of sample paths) were studied in great detail. For an overview we refer to \cite{Khos} and \cite{Walsh} and the references given therein.
The relation of the Brownian sheet with approximation theory also attracted a lot of attention in connection with the probability estimates of small balls, cf. \cite{BL1, BLV1, Dunker, KL02, KL93, KL93', Linde99, Tal94}. However, this problem in its full generality remains unsolved up to now.

\subsection{L\'evy's decomposition of the Brownian sheet}\label{subsec-31}

We now present the decomposition of the paths of the Brownian sheet into the corresponding Faber system. For this sake, we need to
develop the multivariate analogues of the tools of Section \ref{subsec:21}, i.e., multivariate Faber systems and second order differences
of functions of several variables.
We restrict ourselves to the two-dimensional setting, i.e., to the case $d=2$. This simplifies to some extent the notation used, nevertheless,
the general case $d\ge 2$ can be treated in the same way. 
Hence, we now focus on the Brownian sheet $B=(B(x))_{x\in Q}$ on $Q=[0,1]^2$.

Similarly to \eqref{eq:Levy01}, we show that the paths of the Brownian sheet can be decomposed into the multivariate Faber system and that
the coefficients of this decomposition are again independent Gaussian variables.
In order to do this, we follow \cite[Section 3.2]{Triebel10} and  first  describe the decomposition
of (continuous) functions into the two-dimensional Faber system.

\subsubsection{{Multivariate Faber system and second order differences}}

The \emph{multivariate Faber system} is obtained by considering the tensor products of the functions from the univariate Faber system on $[0,1]$, which
was introduced in Section \ref{sec:Faber} as
\begin{equation}\label{eq:Faberd1}
\{v_0(t)=1-t, v_1(t)=t, v_{j,m}(t):j\in \N_0, m=0,\dots,2^j-1\},\quad t\in[0,1].
\end{equation}
In the sequel we use the notation $\N_{-1}:=\N_0\cup \{-1\}=\{-1,0,1,\dots\}$ and put $v_{-1,0}(t)=v_0(t)$ and $v_{-1,1}(t)=v_1(t).$
We consider the tensor products of the functions in \eqref{eq:Faberd1}
\[
v_{k,m}(t_1,t_2)=v_{k_1,m_1}(t_1)\cdot v_{k_2,m_2}(t_2), \quad (t_1,t_2)\in [0,1]^2,
\]
where $k=(k_1,k_2)$  with  $k_i \in\N_{-1}$ 
and $m=(m_1,m_2)$ with  $m_i\in\{0,1\}$ if $k_i=-1$ and $m_i\in\{0,1,\dots,2^{k_i}-1\}$ if $k_i\in \N_0$ for $i=1,2$.
Moreover, by  ${\mathcal P}^F_k$ we denote the admissible set of $m$'s for given $k$. 
Finally, the system
\[
\{v_{k,m}:k\in \N_{-1}^2, m\in {\mathcal P}^F_k\}
\]
is the Faber system on $[0,1]^2.$ 

Similarly to \eqref{eq:triebel1}, the coefficients of the decomposition of a continuous $f\in C(Q)$ will be the \emph{second order differences} of $f$.
These are defined in a rather straightforward manner.
\begin{dfn}\label{dfn:mixeddifference}
\begin{enumerate}
\item If $f$ is a continuous function on $\R^2$, we define second order differences 
\begin{align*}
\Delta^2_{h,1}f(t_1,t_2)&:=f(t_1+2h,t_2)-2f(t_1+h,t_2)+f(t_1,t_2)\\
&=\Delta^1_{h,1}(\Delta^1_{h,1}f)(t_1,t_2)=\sum_{i=0}^2(-1)^i\binom{2}{i}f(t_1+ih,t_2)
\end{align*}
and similarly for $\Delta^2_{h,2}f$. The second order mixed differences are defined as
\begin{align*}
\Delta^{2,2}_{h_1,h_2}f(t_1,t_2)&:=\Delta^2_{h_2,2}(\Delta^2_{h_1,1}f)(t_1,t_2)\\
&=\sum_{i,j=1}^2 (-1)^{i+j}\binom{2}{i}\binom{2}{j}f(t_1+ih_1,t_2+ih_2).
\end{align*}
\item 
For $m\in{\mathcal P}^F_k$, we put 
\begin{equation}\label{eq:d2f}
d^2_{k,m}(f):=\begin{cases}
f(m_1,m_2), & \text{if } k=(-1,-1),\\
-\frac 12 \Delta^2_{2^{-k_2-1},2}f(m_1,2^{-k_2}m_2), & \text{if }  k=(-1,k_2), \ k_2\in \N_0, \\
-\frac 12 \Delta^2_{2^{-k_1-1},1}f(2^{-k_1}m_1,m_2), & \text{if }  k=(k_1,-1), \ k_1\in \N_0, \\
\frac 14 \Delta^{2,2}_{2^{-k_1-1},2^{-k_2-1}}f(2^{-k_1}m_1,2^{-k_2}m_2), & \text{if }  k=(k_1,k_2), \ k_1, k_2\in \N_0.  
\end{cases}
\end{equation}
\end{enumerate}
\end{dfn}
The two-dimensional analogue of Theorem \ref{thm:triebel1} then reads as follows.
\begin{thm}(\cite[Thm.~3.10]{Triebel10})
For $f\in C([0,1]^2)$ it holds
\begin{equation}\label{eq:Triebel11}
f(t)=\sum_{k\in \N^2_{-1}}\sum_{m\in{\mathcal P}_k^F}d^2_{k,m}(f)v_{k,m}(t)=\lim_{K\to\infty}\sum_{k\in\{-1,0,\dots,K\}^2}\sum_{m\in{\mathcal P}_k^F}d^2_{k,m}(f)v_{k,m}(t),\quad t\in[0,1]^2,
\end{equation}
where the limit is taken in the uniform norm.
\end{thm}

\subsubsection{Decomposition of paths of the Brownian sheet}

Similarly to \eqref{eq:Levy01}  we can apply \eqref{eq:Triebel11}   to paths of the Brownian sheet $B=(B(x))_{x\in Q}$
if we replace the scalars $d^2_{k,m}(f)$ by random variables $d^2_{k,m}(B)$. First, observe that \eqref{Brownian-sheet-1}
ensures that $B(0,t)=B(s,0)=0$ almost surely for every $s,t\ge 0$, which implies that almost surely we also have  
\begin{equation}\label{eq:d2zero}
d^2_{k,m}(B)=0\quad\text{if}\quad \begin{cases}k=(-1,-1)\quad\text{and}\quad m=(m_1,m_2)\in\{(0,0),(0,1),(1,0)\},\\
k=(-1,k_2)\quad \text{if}\quad k_2\in\N_0\quad\text{and}\quad m_1=0,\\
k=(k_1,-1)\quad \text{if}\quad k_1\in\N_0\quad\text{and}\quad m_2=0.
\end{cases}
\end{equation}
Furthermore, all the random variables $d^2_{k,m}(B)$ are Gaussian with mean zero. If $k_1=-1$ or $k_2=-1$, then their variance can be computed directly as
\begin{align*}
\var d^2_{(-1,-1),(1,1)}(B)&=\var B(1,1)=1,\\
\var d^2_{(-1,k_2),(1,m_2)}(B)
&=\var\! \Bigl(-\frac 12 \Delta^2_{2^{-k_2-1},2}B(1,2^{-k_2}m_2)\Bigr)\\
&=\frac{1}{4} \var\! \Bigl( B(1,2^{-k_2}(m_2+1))-2B(1,2^{-k_2}(m_2+1/2))+B(1,2^{-k_2}m_2)\Bigr)\\
&=2^{-k_2-2},
\end{align*}
and similarly, we obtain $\displaystyle \var d^2_{(k_1,-1),(m_1,1)}(B)= 2^{-k_1-2}$.

In order to  calculate the variance of $d^2_{k,m}(B)$ for $k_1,k_2\in \N_0$, we introduce some  further notation.
For a cube $\tilde Q=[s_1,s_2]\times[t_1,t_2]$, where $0\le s_1\le s_2\le 1$ and $0\le t_1\le t_2\le 1$, we put  
\begin{align}\label{eq:def_BQ}
B(\tilde Q)&:=B{(s_2,t_2)}-B{(s_1,t_2)}-B{(s_2,t_1)}+B{(s_1,t_1)}\\
\notag&=\Delta^1_{s_2-s_1,1}(\Delta^1_{t_2-t_1,2}B)(s_1,t_1)=\sum_{i=1}^2\sum_{j=1}^2 (-1)^{i+j}B(s_i,t_j).
\end{align}
Then $B(\tilde Q)$ is a centered Gaussian variable with variance
\begin{align*}
\mathrm{var} B(\tilde Q)  &=\E B(\tilde Q)^2= \sum_{i,j,k,l=1}^2 (-1)^{i+j+k+l}\E \left[B(s_i,t_j)B(s_k,t_l)\right] \\
& = \sum_{i,j,k,l=1}^2 (-1)^{i+j+k+l}\min(s_i,s_k)\min(t_j,t_l)\\
&= \sum_{i,k=1}^2 (-1)^{i+k}\min(s_i,s_k) \sum_{j,l=1}^2 (-1)^{j+l}\min(t_j,t_l)=(s_2-s_1)(t_2-t_1), 
\end{align*}
where in the third step we used \eqref{Brownian-sheet-1}. 
Furthermore, if $\tilde Q_1=[s_1,s_2]\times [t_1,t_2]$ and $\tilde Q_2=[\sigma_1,\sigma_2]\times[\tau_1,\tau_2]$ are disjoint, then we obtain in the same way
\begin{align*}
\E B(\tilde Q_1)B(\tilde Q_2)
&= \sum_{i,k=1}^2 (-1)^{i+k}\min(s_i,\sigma_k) \sum_{j,l=1}^2 (-1)^{j+l}\min(t_j,\tau_l)=0.
\end{align*}
Actually, $B(\tilde Q_1)$ and $B(\tilde Q_2)$ are not only uncorrelated but also independent, cf. \cite[Sect.~2.4, Prop.~1]{dal04}.
For $k\in\N_0^2$ and $m\in{\mathcal P}^F_k$, we define
\[
Q_{k,m}=I_{k_1,m_1}\times I_{k_2,m_2}=[2^{-k_1}m_1,2^{-k_1}(m_1+1)]\times [2^{-k_2}m_2,2^{-k_2}(m_2+1)]
\]
and using this notation, we compute for $k_1,k_2\ge 0$, 
\begin{align}
\notag4d^2_{k,m}(B)&=\Delta^{2,2}_{2^{-k_1-1},2^{-k_2-1}}B(2^{-k_1}m_1,2^{-k_2}m_2)\\
\label{eq:4d2}&=B(Q_{(k_1+1,k_2+1),(2m_1+1,2m_2+1)}) - B(Q_{(k_1+1,k_2+1),(2m_1+1,2m_2)})\\
\notag&\qquad- B(Q_{(k_1+1,k_2+1),(2m_1,2m_2+1)}) + B(Q_{(k_1+1,k_2+1),(2m_1,2m_2)}). 
\end{align}
Since the four summands on the right hand side are independent Gaussian variables with variance $2^{-(k_1+k_2)-2}$ 
we obtain that $d^2_{k,m}(B)$ is a Gaussian variable with variance $ 2^{-(k_1+k_2+4)} $.

Merging all what we said about $d^2_{k,m}(B)$ so far, we arrive at the decomposition of the paths of the Brownian sheet into the  multivariate Faber system.
The discussion of independence is postponed to Section \ref{subsec:314}.
\begin{thm}\label{thm:Levy_sheet} For the Brownian sheet on $[0,1]^2$ it holds almost surely that
\begin{align}
\notag B(t_1,t_2)&=\xi_{(-1,-1),(1,1)}t_1t_2
+\sum_{k_1=0}^\infty\sum_{m_1=0}^{2^{k_1}-1} 2^{-(k_1+2)/2}\xi_{(k_1,-1)(m_1,1)}
v_{k_1,m_1}(t_1)t_2\\
\label{eq:Sheet02}&\quad+\sum_{k_2=0}^\infty\sum_{m_2=0}^{2^{k_2}-1} 2^{-(k_2+2)/2}\xi_{(-1,k_2)(1,m_2)}
t_1v_{k_2,m_2}(t_2)\\
\notag &\quad+\sum_{k\in\N_0^2}\sum_{m_1=0}^{2^{k_1}-1}\sum_{m_2=0}^{2^{k_2}-1}2^{-(k_1+k_2+4)/2}\xi_{k,m}v_{k,m}(t_1,t_2),
\end{align}
where $\left\{\xi_{k,m}: \ k\in \N_{-1}^2, \ m\in {\mathcal P}^F_k\right\}$ are independent standard Gaussian variables and the series converges uniformly on $[0,1]^2$.
\end{thm}
\begin{rem}
If $k\in\N_{-1}^2$ and $m\in{\mathcal P}_k^F$, then we denote $\gamma_{k,m}=\gamma_{(k_1,k_2),(m_1,m_2)}=\gamma_{k_1,m_1}\cdot\gamma_{k_2,m_2}$, where
\[
\gamma_{j,l}=\begin{cases}0& \text{if}\quad j=-1\ \text{and}\ l=0,\\
1& \text{if}\quad j=-1\ \text{and}\ l=1,\\
2^{-(j+2)/2}& \text{if}\quad j\ge 0.
\end{cases}
\]
This allows us to reformulate \eqref{eq:Sheet02} as
\begin{equation}\label{eq:Sheet01}
B(t_1,t_2)=\sum_{k\in\N_{-1}^2}\sum_{m\in {\mathcal P}_k^F}\gamma_{k,m}\xi_{k,m}v_{k,m}(t_1,t_2).
\end{equation}
\end{rem}

\subsubsection{Independence}\label{subsec:314}
Next, we show that the random variables $\{d_{k,m}^2(B):k\in\N_{-1}^2, m\in{\mathcal P}^F_k\}$, which appear  in Theorem \ref{thm:Levy_sheet}, are indeed independent. 
The argument is a tensor product variant of the proof given for the univariate Wiener process, cf. Theorem \ref{thm:LevyDecomp}.
For that sake, let us consider mutually different $(k^1,m^1),\dots,(k^N,m^N)$ with $k^i=(k^i_1,k^i_2)$ and $m^i=(m^i_1,m^i_2)$. Let
\[
K=(K_1,K_2),\quad \text{where}\ K_j=\max\{k^1_j,\dots,k^N_j\} \ \text{ for }\ j\in\{1,2\}, 
\]
and let us consider the array of Gaussian variables
\[
\widetilde B^K=(B(Q_{K+1,m}):m=(m_1,m_2)\ \text{and}\ 0\le m_j\le 2^{K_j}-1\ \text{for}\ j\in\{1,2\}),
\]
where $B(\tilde Q)$ was defined in \eqref{eq:def_BQ} for a closed cube $\tilde Q\subset [0,1]^2$ with sides parallel to the coordinate axes.

Moreover, let $I,J\subset [0,1]$ be two (closed) intervals,  where  $I=I_1\cup I_2$ and $J=J_1\cup J_2$ are   decompositions of $I$ and $J$ into
two intervals $I_1$, $I_2$ and $J_1$, $J_2$, respectively, which intersect only at one point. Then a straightforward calculation shows that
\begin{align}
\label{eq:Bnotation_split1}B(I\times J)&=B((I_1\cup I_2)\times J)=B(I_1\times J)+B(I_2\times J)\\
\label{eq:Bnotation_split2}B(I\times J)&=B(I\times (J_1\cup J_2))=B(I\times J_1)+B(I\times J_2).
\end{align}
Furthermore, we denote
\begin{equation}\label{eq:tensor_h}
\bigl(h^{(K_1,K_2)}_{(k_1,k_2),(m_1,m_2)}\bigr)_{(l_1,l_2)}=
\bigl(h^{K_1}_{k_1,m_1}\bigr)_{l_1}\cdot \bigl(h^{K_2}_{k_2,m_2}\bigr)_{l_2},
\end{equation}
where the vectors on the right-hand side were defined in \eqref{eq:def_h} for $k_1,k_2\ge 0$ and we complement this definition by putting $h^{K_j}_{-1,1}=(1,\dots,1)^T$ for $j=1,2$. 

If $k=(k_1,k_2)=(-1,-1)$, we observe that by \eqref{eq:d2zero} it is enough to consider $m=(1,1).$
Then we apply \eqref{eq:d2f}, \eqref{eq:Bnotation_split1} and \eqref{eq:Bnotation_split2} together with \eqref{eq:tensor_h}
and obtain
\[
d^2_{(-1,-1),(1,1)}(B)=B(Q)=\sum_{l_1=0}^{2^{K_1+1}-1}\sum_{l_2=0}^{2^{K_2+1}-1}{B(Q_{K+1,l})}= \langle h^K_{k,m},\widetilde B^K\rangle.
\]

If $k=(k_1,k_2)=(-1,k_2)$ with $0\le k_2\le K_2$, we assume by \eqref{eq:d2zero} that $m=(1,m_2)$, where $0\le m_2\le 2^{k_2}-1$.
Then we combine \eqref{eq:d2f}, \eqref{eq:Bnotation_split1}, \eqref{eq:Bnotation_split2}, and \eqref{eq:tensor_h} and obtain
\begin{align*}
-2d^2_{(-1,k_2),(1,m_2)}(B)=B([0,1]\times I_{k_2+1,2m_2+1})-B([0,1]\times I_{k_2+1,2m_2})=\langle h^K_{k,m},\widetilde B^K\rangle.
\end{align*}
The case of $k=(k_1,-1)$ is treated similarly. And finally, if $k=(k_1,k_2)\in \N_0^2$ with $0\le k_1\le K_1$ and $0\le k_2\le K_2$,
we employ \eqref{eq:4d2} and observe that
\[
4d^2_{k,m}=\langle h^K_{k,m},\widetilde B^K\rangle.
\]
The independence of $\bigl(d^2_{k^i,m^i}(B)\bigr)_{i=1}^N$ now follows again by the orthogonality of $\bigl(h^K_{k^i,m^i}\bigr)_{i=1}^N$,
which in turn is a consequence of \eqref{eq:tensor_h}.

\subsection{Function spaces}

As already mentioned   before, the appropriate function spaces, which best capture the regularity of the paths of the Brownian sheet, are the function spaces of dominating mixed smoothness. They appeared for the first time in the work of Babenko \cite{Babenko}, but they are also known to play an important role   in approximation theory and numerics of PDE's \cite{BG04,DTU18,SU09,Teml93}. In this context we also refer   to \cite{HansenDiss,SchTr87, Triebel12, Triebel19,Vyb06} (and the references given therein) for a comprehensive treatment.

Similarly to the univariate case, we shall need several different variants of function spaces of dominating mixed smoothness. As before, we start with the spaces of Besov type.

\begin{dfn}
Let $0<r<l\in \N$ and $1\leq p,q\leq \infty$. Then $S^r_{pq}B(\R^2)$ is the collection of all $f\in L_p(\R^2)$ such that 
\begin{align*}
\|f|L_p(\R^2)\|
&+ \left(\int_0^1 t^{-rq}\sup_{|h_1|\leq t}\|\Delta^l_{h_1,1}f|L_p(\R^2)\|^2\frac{dt}{t}\right)^{1/q}\\
&+ \left(\int_0^1 t^{-rq}\sup_{|h_2|\leq t}\|\Delta^l_{h_2,2}f|L_p(\R^2)\|^2\frac{dt}{t}\right)^{1/q}\\
&+  \left(\int_0^1\int_0^1 (t_1t_2)^{-rq}
\sup_{\substack{|h_1|\leq t_1},{|h_2|\leq t_2}}
\|\Delta^{l,l}_{h_1,h_2}f|L_p(\R^2)\|^q\frac{dt_1 dt_2}{t_1t_2}\right)^{1/q}
\end{align*}
is finite. Here we use the usual mixed version of differences resulting from Definition \ref{dfn:mixeddifference}, i.e., 
\begin{align*}
    \Delta^{l,l}_{h_1,h_2}f(t_1,t_2)=\Delta^l_{h_2,2}\left(\Delta^l_{h_1,1}f\right)(t_1,t_2)\quad\text{ and }\quad\Delta^{l+1}_{h_i,i}f(x)=\Delta^{1}_{h_i,i}\left(\Delta^{l}_{h_i,i}f\right)(x).
\end{align*}
\end{dfn} 


The restriction of the spaces from $\R^2$ to $Q=[0,1]^2$ is done  in the same way as described in Subsection \ref{sec:Faber}, cf. \eqref{restriction}. 

We now provide a decomposition of functions $f$ from $S^r_{pq}B(Q)$ via the higher dimensional Faber system introduced in Subsection \ref{subsec-31}. 
This follows from  Theorem 3.16 of \cite{Triebel10} and its extension  \cite[Thm.~4.25]{Byrenheid}, which  ensure the following two-dimensional counterpart of Theorem \ref{thm:faber_s}, cf. also \cite[Theorem A]{Kamont2}. 

\begin{thm}\label{Faber-domi}  Let  $0<p,q\le\infty$, $p>\frac 12$, and
\begin{equation*}
\frac{1}{p}<r<1+\frac 1p
\end{equation*}
be the admissible range for $r$. Then $f\in L_1(Q)$ lies in $S^r_{p,q}B(Q)$ if, and only if, it can be represented (with convergence in $L_1(Q)$) as
\begin{align}\label{eq:Triebel01a}
 f&=\sum_{k\in\N_{-1}^2}\sum_{m\in{\mathcal P}_k^F} \lambda_{k,m}2^{-(k_1+k_2)r}v_{k,m},
\intertext{with}
\label{eq:def_lambda_1}
\|\lambda|s_{p,q}^Fb(Q)\|&=\Bigg(\sum_{k\in \N_{-1}^2}\Big(\sum_{m\in{\mathcal P}^F_k}2^{-(k_1+k_2)}|\lambda_{k,m}|^p\Big)^{q/p}\Bigg)^{1/q}<\infty.
\end{align}
Furthermore, the representation \eqref{eq:Triebel01a} is unique with
\begin{align*}
    \lambda_{k,m}=\lambda_{k,m}(f)=2^{(k_1+k_2)r}d^2_{k,m}(f),\quad k\in\N_{-1}^2, \,m\in{\mathcal P}_k^F.
\end{align*}
\end{thm}
The dominating mixed smoothness counterpart of the one dimensional spaces $B^{s,\alpha}_{p,q}(I)$ with logarithmic smoothness has been studied  in \cite{Triebel10}.
Unfortunately, we can not rely exclusively on the results of \cite{Triebel10} for two reasons. First,  the characterization with the Faber system presented in Theorem 3.35 in \cite{Triebel10} does not include the case $q=\infty$, which is important for our considerations. Furthermore, it will turn out later, that the way we introduce the logarithmic smoothness differs from \cite{Triebel10}. Indeed, the factor $(k_1\cdot k_2)^{-\alpha}$ used in this reference, gets replaced by $(k_1+k_2)^{-\alpha}$, which is strictly larger for $k=(k_1,k_2)\in\N^2$ and $\alpha>0$.

However, the function spaces of dominating logarithmic smoothness, which we introduce in Definition \ref{dfn:mixed:LogOrlicz}, coincide with the ones used in \cite{Kamont1}.
There these spaces were defined by differences and moduli of smoothness. Moreover, in  \cite[Lemma 3.1]{Kamont1} (based  \cite{Kamont2}) the author further obtained isomorphisms between these function spaces and corresponding sequence spaces. An equivalent Fourier-analytic characterization of these spaces still seems to be  missing. In order to avoid the technicalities we \emph{define} the function spaces by posing a condition on the coefficients appearing in the Faber system expansion \eqref{eq:Triebel01a}.

Essentially the same comment applies to Besov-Orlicz spaces of dominating mixed smoothness. Except \cite{Kamont1}, we are not aware of 
any existing work, which introduces and systematically studies these function spaces.
Nevertheless, comparing the one dimensional case with Theorem \ref{Faber-domi} and Theorem 3.35 in \cite{Triebel10}, the following definition seems to be well-motivated and a natural generalization. Again we restrict ourselves to the case of parameters which we need later on, i.e. $q=\infty$ and $r=1/2$.

Let us recall, that $\chi_{j,l}$ was defined as the characteristic function of $I_{j,l}$ for $j\ge 0$. We complement this notation by $I_{-1,0}=I_{-1,1}=[0,1]$ and put
\[
\chi_{k,m}(t_1,t_2)=\chi_{k_1,m_1}(t_1)\cdot \chi_{k_2,m_2}(t_2),\quad k=(k_1,k_2)\in\N_{-1}^2,\ t=(t_1,t_2)\in Q.
\]

\begin{dfn}\label{dfn:mixed:LogOrlicz}
\begin{enumerate}
    \item Let $0<p\leq\infty$ and $\alpha\in\R$.
    The function space $S^{1/2,\alpha}_{p,\infty}B(Q)$ is the collection of all $f\in C(Q)$ which can be represented by \eqref{eq:Triebel01a} with $r=1/2$
and
\begin{align}\label{eq:def_lambda_2}
\|\lambda|s_{p,\infty}^{F,\alpha}b(Q)\|&=\sup_{k\in \N_{-1}^2}
{\max(1,k_1+k_2)^{-\alpha}}
\Big(\sum_{m\in{\mathcal P}^F_k}2^{-(k_1+k_2)}|\lambda_{k,m}|^p\Big)^{1/p}<\infty.
\end{align}
    \item The function space $S^{1/2}_{{\Phi_2},\infty}B(Q)$ is the collection of all $f\in C(Q)$ which can be represented by \eqref{eq:Triebel01a} with $r=1/2$ and
\begin{align}\label{eq:def_lambda_3}
\|\lambda|s_{\Phi_2,\infty}^{F}b(Q)\|&=\sup_{k\in \N_{-1}^2}\left\|\sum_{m\in{\mathcal P}^F_k}\lambda_{k,m}\chi_{k,m}(\cdot)\right\|_{\Phi_2}
<\infty.
\end{align}
\end{enumerate}
\end{dfn}
\begin{rem}
If in \eqref{eq:def_lambda_2} we had  used  the usual logarithmic Hölder spaces from \cite[Theorem 3.35]{Triebel10} the sequence space norm would be  
\begin{align*}
    \sup_{k\in \N_{-1}^2}
(2+k_1)^{-\alpha}(2+k_2)^{-\alpha}
\Big(\sum_{m\in{\mathcal P}^F_k}2^{-(k_1+k_2)}|\lambda_{k,m}|^p\Big)^{1/p}<\infty.
\end{align*}
Compared to this the advantage of our approach in \eqref{eq:def_lambda_2} is, that the sequence spaces are strictly smaller for $\alpha>0$ and we therefore obtain   better regularity results. The disadvantage on the other hand is,  that the smoothness weights in \eqref{eq:def_lambda_2} do not have any tensor product structure.
\end{rem}
\begin{OP}
For decades the role of function spaces of dominating mixed smoothness $S^r_{p,q}B(Q)$ in approximation theory has been studied intensively. We refer to \cite{DTU18} for a recent overview on various  results and a list of open problems in this field. Compared to that, much less seems to be known regarding  function spaces with logarithmic smoothness and Besov-Orlicz spaces as introduced before. In our opinion, it would be worth investigating these spaces together with their embeddings, which in turn might  shed   new light on (some of) the open problems in \cite{DTU18}.
\end{OP}
\subsection{Results of Kamont}

We now recover the results of Kamont \cite{Kamont1} (cf. also \cite{Hummel}) on the regularity of the sample paths of Brownian sheets. We combine the representation in \eqref{eq:Sheet01} (or \eqref{eq:Sheet02}) with Theorem \ref{Faber-domi} and Definition \ref{dfn:mixed:LogOrlicz} in order to show that the sample paths lie in $S^{1/2,1/2}_{\infty,\infty}B(Q)$, $S^{1/2}_{p,\infty}B(Q)$ for all $1\le p< \infty$, and $S^{1/2}_{\Phi_2,\infty}B(Q)$ almost surely. Similar to the method used in  Section \ref{sub:2.3}, we just need to verify that the condition on the coefficients in the Faber system decomposition is fulfiled almost surely, if we replace $\lambda$ in \eqref{eq:def_lambda_1},
\eqref{eq:def_lambda_2}, and \eqref{eq:def_lambda_3} by a sequence of independent normal variables $\xi$.

For the terms with $k_1=-1$ or $k_2=-1$, the conditions \eqref{eq:def_lambda_1}, \eqref{eq:def_lambda_2}, and \eqref{eq:def_lambda_3} reduce to their one-dimensional counterparts, which were already discussed in Section \ref{sub:2.3}, cf. also \eqref{eq:d2zero}. Furthermore, the same is true for the terms with $k_1=0$ or $k_2=0$. Therefore, it will be enough to handle the terms with $k=(k_1,k_2)\in\N^2.$

\noindent {\bf 1. Paths of the Brownian sheet lie almost surely in $S^{1/2,1/2}_{\infty,\infty}B(Q)$}\\
In view of Definition \ref{dfn:mixed:LogOrlicz}, we need to show that
\begin{gather*}
\sup_{k\in\N^2}\frac{1}{\sqrt{k_1+ k_2}}\sup_{m_1=0,\dots,2^{k_1}-1}\sup_{m_2=0,\dots,2^{k_2}-1}|\xi_{k,m}|<\infty\quad\text{almost surely},
\end{gather*}
where $\{\xi_{k,m}: k\in\N^2, 0\le m_1\le 2^{k_1}-1, 0\le m_2\le 2^{k_2}-1\}$ are independent standard Gaussian variables.
For this sake, we define the event $A^N_{k}$ as
\[
\sup_{m_1=0,\dots,2^{k_1}-1}\sup_{m_2=0,\dots,2^{k_2}-1}|\xi_{k,m}|>N\sqrt{k_1+ k_2}.
\]
Similar to \eqref{eq:AtN_upper1}, we obtain that
\[
\P(A^N_{k})\le 2^{k_1+k_2}e^{-N^2(k_1+k_2)/2},\quad k\in\N^2,
\]
and following \eqref{eq:sim1},  for every $N_0\ge 1$ we get
\begin{align*}
\P&\biggl(\sup_{k\in\N^2}\frac{1}{\sqrt{k_1+ k_2}}\sup_{m_1=0,\dots,2^{k_1}-1}\sup_{m_2=0,\dots,2^{k_2}-1}|\xi_{k,m}|=\infty\biggr)=\P\Bigl(\bigcap_{N=1}^\infty\bigcup_{k \in \N^2} A_k^N\Bigr)\\
&\le \P\Bigl(\bigcup_{k \in \N^2} A_k^{N_0}\Bigr)\le \sum_{k \in \N^2} \P\bigl(A_k^{N_0}\bigr)\le \sum_{k \in \N^2}2^{k_1+k_2}e^{-N_0^2(k_1+k_2)/2}.
\end{align*}
As the last sum again tends to zero if $N_0\to\infty$, the proof is finished. 

\noindent {\bf 2. Paths of the Brownian sheet lie almost surely in $S^{1/2}_{p,\infty}B(Q)$ for every $1\le p<\infty$}

In order to be able to apply Theorem \ref{Faber-domi} for $r=1/2$, we assume that $2<p<\infty$. Then the smaller values of $p$ follow easily by monotonicity of the function spaces with dominating mixed smoothness on domains with respect to the integrability parameter $p$. In this case  \eqref{eq:def_lambda_1} for $k\in\N^2$ reduces to 
\begin{gather*}
\sup_{k\in\N^2}\frac{1}{2^{k_1+k_2}}\sum_{m_1=0}^{2^{k_1}-1}\sum_{m_2=0}^{2^{k_2}-1}|\xi_{k,m}|^p<\infty\quad\text{almost surely.}
\end{gather*}
We denote again by $\mu_{p}$ the $p$ absolute moment of a standard Gaussian variable.
Furthermore, for $t>0$ and $k\in\N^2$, we denote by $A^t_k$ the event
\[
\frac{1}{2^{k_1+k_2}}\sum_{m_1=0}^{2^{k_1}-1}\sum_{m_2=0}^{2^{k_2}-1}|\xi_{k,m}|^{p}-\mu_{p}\ge t.
\]
By \eqref{eq:Ajt} applied to $j=k_1+k_2$, we observe that
\[
\P(A^t_k)\le \frac{1}{t^2}\cdot \frac{\mu_{2p}-\mu_{p}^2}{2^{k_1+k_2}}.
\]  
Finally, we conclude, that for every $N_0\in\N$ it holds
\begin{align*}
\P\biggl(\sup_{k\in\N^2}\frac{1}{2^{k_1+k_2}}\sum_{m_1=0}^{2^{k_1}-1}\sum_{m_2=0}^{2^{k_2}-1}|\xi_{k,m}|^p=\infty\biggr)&=\P\Bigl(\bigcap_{N=1}^\infty\bigcup_{k \in \N^2} A_k^N\Bigr)
\le \P\Bigl(\bigcup_{k \in \N^2} A_k^{N_0}\Bigr)\le \sum_{k \in \N^2} \P\bigl(A_k^{N_0}\bigr)\\
&\le \frac{\mu_{2p}-\mu_p^2}{N_0^2}\sum_{k_1=1}^{\infty}\sum_{k_2=1}^{\infty} \frac{1}{2^{k_1+k_2}}=\frac{4(\mu_{2p}-\mu_p^2)}{N_0^2}.
\end{align*}
The last expression tends to zero if $N_0\to\infty$, which renders the result.

\noindent {\bf 3. Paths of the Brownian sheet lie almost surely in $S^{1/2}_{\Phi_2,\infty}B(Q)$}\\
Again, it is enough to show that \eqref{eq:def_lambda_3} is finite almost surely if we restrict ourselves to $k\in\N^2$ and replace $\lambda_{k,m}$ by independent normal variables $\xi_{k,m}$. Therefore, we put 
\begin{align}\label{eq:fk2}
f_k(t)=\sum_{m_1=0}^{2^{k_1}-1}\sum_{m_2=0}^{2^{k_2}-1}
\xi_{k,m}{\chi_{k,m}(t)},\quad k=(k_1,k_2)\in\N^2\quad\text{and}\quad t\in Q
\end{align}
and show that
\begin{align*}
\sup_{k\in\N^2}\|f_k\|_{\Phi_2}<\infty\quad\text{almost surely}.
\end{align*}
Again we  use the characterization of $L_{\Phi_2}(Q)$ with the non-increasing rearrangement from Theorem \ref{thm:Orlicz}. As in Section \ref{sub:2.4} we have 
\begin{align}\label{eq:fkstern}
    f^*_k(s)=\sum_{m=0}^{2^{k_1+k_2}-1}(\xi_k)^*_{m+1}\chi_{k_1+k_2,m}(s)\quad\text{for }0<s<1
\end{align}
and using \eqref{eq:flog_bound} we can estimate
\[
\|f_k\|_{\Phi_2}\le c\sup_{0\le u<k_1+k_2}\frac{f_k^*(2^{u-(k_1+k_2)})}{\sqrt{k_1+k_2-u}}
=c\sup_{0\le u<k_1+k_2}\frac{(\xi_k)^*_{2^u}}{\sqrt{k_1+k_2-u}},
\]
where $\xi_k=(\xi_{k,m}:0\le m_1\le 2^{k_1}-1, 0\le m_2\le 2^{k_2}-1)$
and $\left((\xi_{k})^*_m\right)_{m=1}^{2^{k_1+k_2}}$ is its non-increasing rearrangement. Then, by Lemma \ref{lem:tail_bounds} we get for every $K_0$ large enough
\begin{align*}
\P\Bigl(\sup_{k\in\N^2}\|f_k\|_{\Phi_2}&=\infty\Bigr)\le \sum_{0\le u<k_1+k_2<\infty}\Bigl(2e^{-K_0^2/2}\Bigr)^{(k_1+k_2-u)2^u}\cdot e^{2^u}\\
&\le c\sum_{u=0}^\infty e^{2^u}\sum_{l=u+1}^\infty l\cdot\Bigl(2e^{-K_0^2/2}\Bigr)^{(l-u)\cdot 2^u}\le c\sum_{u=0}^\infty (u+1)e^{2^u}\cdot\Bigl(2e^{-K_0^2/2}\Bigr)^{2^u},
\end{align*}
which again goes to zero if $K_0\to\infty$.
\begin{rem}
At this point, we can observe how versatile the method with the non-increasing rearrangement is performing. As input we have given a two dimensional function in \eqref{eq:fk2} and using the rearrangement we obtain in \eqref{eq:fkstern} a one-dimensional function similar to the one from  Section \ref{sub:2.4}. Now the remaining estimates done above are an analogue of the estimates of the case $d=1$ with the only difference that $k$ is replaced by $k_1+k_2$.

Clearly, a generalization to  $d>2$ follows directly from the considerations above - as long as a  characterization in terms of a Faber basis as in Theorem \ref{thm:Levy_sheet} is given.
\end{rem}

\subsection{New function spaces}
We also can carry over the one-dimensional approach of averaging operators  to the two-dimensional setting. To that end, we add to \eqref{eq:AV1} the extra case
\begin{align*}
    (A_{-1}g)(x)=g(x)\quad\text{for all }x\in I.
\end{align*}
Now, for an integrable function $g$ on $Q$ we can introduce  the averaging operators $A_kg$ for all $k=(k_1,k_2)\in\N_{-1}^2$ by
\begin{align*}
    (A_kg)(x):=A^1_{k_1}\left(A^2_{k_2}g(\cdot,x_2)\right)(x_1)\quad \text{ and } \quad  (\tilde{A}_kg)(x)=(A_k(|g|))(x), \quad x\in Q,
\end{align*}
where the upper index stands for the dimension in which the one-dimensional averaging operator is applied. In the non-borderline case when  $k\in\N_0^2$, the operator is given by
\begin{align*}
    (A_kg)(x)=\sum_{l_1=0}^{2^{k_1}-1}\sum_{l_2=0}^{2^{k_2}-1}2^{k_1+k_2}\int_{Q_{k,l}}g(t)dt\cdot\chi_{k,l}(x), 
\end{align*}
where $k\in\N_0^2$,  $Q_{k,l}=I_{k_1,l_1}\times I_{k_2,l_2}$,  and $\chi_{k,l}(t_1,t_2)=\chi_{k_1,l_1}(t_1)\chi_{k_2,l_2}(t_2)$.\\

With these preparations we can define the spaces $A_{\varepsilon}(Q)$ and $\widetilde{A}(Q)$ as follows.

\begin{dfn}
Let $f\in C(Q)$ be represented in the Faber system by
\begin{align*}
 f&=\sum_{k\in\N_{-1}^2}\sum_{m\in{\mathcal P}_k^F} \lambda_{k,m}2^{-(k_1+k_2)\frac{1}{2}}v_{k,m},
\end{align*}
 and define $f_j$ by 
\begin{align*}
f_j(t)=\sum_{m\in{\mathcal P}_j^F}
\lambda_{j,m}{\chi_{j,m}(t)},\quad j=(j_1,j_2)\in\N^2_{-1}\quad\text{and}\quad t\in Q.
\end{align*}
\begin{enumerate}
    \item For $\varepsilon>0$ the space $A_{\varepsilon}(Q)$ is the collection of all $f\in C(Q)$ with
\begin{align*}
\|f\|_{A_{\varepsilon}}&:=\sup_{j\in\N_{-1}^2}\sup_{-1\le k_1\le j_1}\sup_{-1\le k_2\le j_2} \sup_{0<t<1} \frac{2^{(j_1+j_2-(k_1+k_2))/2}}{(j_1+j_2-(k_1+k_2)+1)^\varepsilon}\cdot\frac{(A_kf_j)^*(t)}{\sqrt{\log(1/t)+1}}\\
&\approx \sup_{j\in\N_{-1}^2}\sup_{-1\le k_i\le j_i} \frac{2^{(j_1+j_2-(k_1+k_2))/2}}{(j_1+j_2-(k_1+k_2)+1)^\varepsilon}\cdot \|A_kf_j\|_{\Phi_2}<\infty.
\end{align*}

\item The space $\widetilde{A}(Q)$ is the collection of all $f\in C(Q)$ with
\begin{align*}
\|f\|_{\widetilde A}&:=\sup_{j\in\N_{-1}^2}\sup_{-1\le k_1\le j_1}\sup_{-1\le k_2\le j_2} \sup_{0<t<1} \frac{(\widetilde A_kf_j)^*(t)}{\sqrt{2^{k_1+k_2-(j_1+j_2)}\log(1/t)+1}}\\
&\approx \sup_{j\in\N_{-1}^2}\sup_{-1\le k_i\le j_i} \|{\widetilde A}_kf_j\|_{2,2^{k_1+k_2-(j_1+j_2)}}<\infty.
\end{align*}
\end{enumerate}
\end{dfn}

Using  this definition we are now able to prove the following result.

\begin{thm}
Let $B$ be the Brownian sheet according to Definition \ref{dfn:BrownianSheet} and $\varepsilon>0$.
\begin{itemize}
    \item[(i)] It holds that $\|B({\cdot})\|_{A_\varepsilon}<\infty$ almost surely.
    \item[(ii)] It holds that $\|B({\cdot})\|_{\widetilde A}<\infty$ almost surely.
\end{itemize}
\end{thm}
\begin{proof}
For both assertions we use that according to  Theorem \ref{thm:Levy_sheet} we have the  representation needed for  the Brownian sheet. Furthermore, we only have to treat the cases where $j=(j_1,j_2)\in\N_0^2$ since the other cases follow directly from the one-dimensional case. For $j=(j_1,j_2)\in\N_0^2$ and $t=(t_1,t_2)\in Q$ we set 
\begin{align*}
f_j(t)=\sum_{m_1=0}^{2^{j_1}-1}\sum_{m_2=0}^{2^{j_2}-1}
\xi_{j,m}{\chi_{j,m}(t)},
\end{align*}
with independent normal variables $\xi_{j,m}$.

In order to  establish the first assertion it is enough to observe that $2^{[j_1+j_2-(k_1+k_2)]\frac{1}{2}}A_kf_j(t)$ is equidistributed as $f_k(t)$, which follows again by the 2-stability, cf. Lemma \ref{2-stability}. Now the proof of this assertion follows directly from \eqref{eq:fkstern} and repeating the same arguments as in the proof of Theorem \ref{thm:Aepsilon} with $k_1+k_2$ and $j_1+j_2$ replacing $k$ and $j$ there.

For the second assertion of the theorem we use
\begin{align*}
(\widetilde{A}_kf_j)(t)=\sum_{l_1=0}^{2^{k_1}-1}\sum_{l_2=0}^{2^{k_2}-1}\nu_{k,l}\chi_{k,l}(t),
\end{align*}
where $\nu_k=(\nu_{k,l}:l_i\in\{0,\dotsc,2^{k_i}-1\})$ is a vector of independent variables from $\mathcal{G}_{2^{j_1+j_2-(k_1+k_2)}}$, see Definition \ref{dfn:G_N}.

Similar to  the proof of Theorem \ref{thm:13}, using \eqref{eq:fkstern}  we estimate with the rearrangement
\begin{align*}
\sup_{0<t<1} &\frac{(\widetilde A_kf_j)^*(t)}{\sqrt{2^{k_1+k_2-(j_1+j_2)}\log(1/t)+1}}
\leq c\,\sup_{z=0,\dots,k_1+k_2-1}\frac{(\nu_k^*)_{2^{z}}}{\sqrt{2^{k_1+k_2-(j_1+j_2)}(k_1+k_2-z-1)+1}},
\end{align*}
where $\left((\nu_k^*)_m\right)_{m=1}^{2^{k_1+k_2}}$ is the rearrangement of $\nu_k$ above. Now, we have reformulated again the two dimensional problem into a one dimensional one with the help of the non-increasing rearrangement and we can use the arguments of the proof of Theorem \ref{thm:13}. 
\end{proof}

\begin{rem}
The proposed analogy with the one-dimensional setting could be investigated even further by defining the analogues of the function spaces $D$ and ${\mathfrak D}$ from Definition \ref{def:DD}. Although the general direction seems to be quite obvious, we do not pursue it  in this paper in order to avoid further technicalities.
\end{rem}

\section{A few facts about random variables and function spaces}\label{sec:app}

\subsection{Gaussian variables}

For the sake of completeness, we present the definition of Gaussian random variables and recall some  of their basic properties.

\begin{dfn}
We say that  the random variable $\xi$  {has  standard normal distribution (or standard Gaussian distribution)} and write $\xi\sim \mathcal{N}(0,1)$, if its density function is given by 
\[
p(x)=\frac{1}{\sqrt{2\pi}}\,e^{-x^2/2},\quad x\in\R. 
\]
\end{dfn}

We will need the following two properties of Gaussian variables (and refer to \cite[Section 3.2]{Tong} for details).

\begin{lem}\label{2-stability}
Let $k\in \mathbb{N}$ and let $\xi=(\xi_1,\ldots, \xi_k)$ be a vector of independent  {standard normal random variables.}
\begin{itemize}
\item[(i)] {\bfseries \upshape (2-stability of normal distribution)}
Let $\lambda=(\lambda_1,\ldots, \lambda_k)\in \mathbb{R}^k$.  Then the random variable 
$\langle\lambda,\xi\rangle=\lambda_1\xi_1+\ldots+\lambda_k\xi_k$ is a normal variable with mean zero and variance $\sum_{i=1}^k\lambda_i^2$.
\item[(ii)] Let $1\le j \le k$ be positive integers and let $u^1,\dots,u^j\in\R^k$ be orthogonal. Then the random variables
$\langle u^1,\xi\rangle,\dots,\langle u^j,\xi\rangle$ are independent.
\end{itemize}
\end{lem}
We shall also make use of the following tail bounds.

\begin{lem}[{\bfseries \upshape Concentration inequalities for standard Gaussian variables}]\label{lem:tail_bounds}\hfill 
\begin{enumerate}
\item[(i)] Let $\omega$ be a standard Gaussian variable. Then
\begin{equation}\label{eq:tailgauss}
\P(|\omega|\ge x)\leq \frac{2\exp(-x^2/2)}{\sqrt{2\pi}x},\quad x>0.
\end{equation}
\item[(ii)] Let $0\le k<j$  be integers and let $\xi=(\xi_0,\dots,\xi_{2^j-1})$ be a vector of independent standard normal random variables. Then, for every $K\ge 1$,
\begin{equation*}
\P\left(\xi^*_{2^k}\ge K\sqrt{j-k}\right)\le \Bigl(2e^{-K^2/2}\Bigr)^{(j-k)2^k}\cdot {e}^{2^k}.
\end{equation*}
Here, $\xi^*=(\xi^*_1,\dots,\xi^*_{2^j})$ is the non-increasing rearrangement of $\xi.$
\end{enumerate}
\end{lem}

\begin{proof}
(i) The proof follows from the elementary bound
\[
\int_x^\infty e^{-u^2/2}du\le \frac{1}{x}\int_x^\infty ue^{-u^2/2}du=\frac{e^{-x^2/2}}{x}.
\]

(ii) Since $\xi_{m}\sim \mathcal{N}(0,1)$ are independent, \eqref{eq:tailgauss} together with the estimate $\displaystyle \binom{n}{k}< \left(\frac{en}{k}\right)^k$
for the binomial coefficients yields  
\begin{align*}
\P\left(\xi^*_{2^k}\ge K\sqrt{j-k}\right)&\le \binom{2^j}{2^k}\P\left(|\omega|\ge K\sqrt{j-k}\right)^{2^k}\le\left(\frac{e2^j}{2^k}\right)^{2^k}
\left(
\frac{2e^{-K^2(j-k)/2}}{\sqrt{2\pi K^2(j-k)}}
\right)^{2^k}\\
&\le \, {e^{2^k}}
\Bigl(2e^{-K^2/2}\Bigr)^{(j-k)2^k}.
\end{align*}
\end{proof}

\subsection{Absolute values of Gaussian variables}

In this section we consider averages of absolute values of Gaussian variables.

\begin{dfn}\label{dfn:G_N}
Let $N\ge 1$ be a positive integer and let $\xi_1,\dots,\xi_N$ be independent standard normal variables. Let $\nu$ be a random variable. We write $\ \nu\sim{\mathcal G}_N\ $ if $\nu$ {has the same distribution as} 
\[
\frac{1}{N}\sum_{m=1}^N |\xi_m|.
\]
\end{dfn}

In terms of  concentration inequalities we have  the following result. 

\begin{lem}[{\bfseries \upshape Concentration inequalities for $\nu$}]\label{lem-prop-gv}\hfill 
\begin{enumerate}
\item[(i)] Let $N\ge 1$, $\nu\sim{\mathcal G}_N$, and $\omega\sim{\mathcal N}(0,1)$. Then
\[
{\P}(\nu\ge t)\le 2^N \P(\omega\ge\sqrt{N} t),\quad t>0.
\]
Moreover, if $t\ge 2\sqrt{\ln 2}$, then
\[
{\P}(\nu\ge t)\le \exp(-Nt^2/4).
\]
\item[(ii)] Let $0\le k<j$ be integers and let $\nu=(\nu_0,\dots,\nu_{2^j-1})$ be a vector of independent ${\mathcal G}_N$ variables. Then
\begin{equation}\label{eq:GN1}
\P(\nu^*_{2^k}\ge t)\le \Bigl[e\cdot 2^{j-k}\exp(-Nt^2/4)\Bigr]^{2^k},\quad t\ge 2\sqrt{\ln 2}.
\end{equation}
\end{enumerate}
\end{lem}
\begin{proof}
(i) We use the 2-stability of normal variables from Lemma \ref{2-stability} and estimate
\begin{align*}
\P(\nu\ge t)&=\P\Bigl(\sum_{m=1}^N|\xi_m|\ge Nt\Bigr)=\P\Bigl(\exists\varepsilon\in\{-1,+1\}^N:\sum_{m=1}^N \varepsilon_m\xi_m\ge Nt\Bigr)\\
&\le 2^N\P\Bigl(\sum_{m=1}^N\xi_m\ge Nt\Bigr)=2^N\P(\omega\ge \sqrt{N}t).
\end{align*}
The second statement follows by \eqref{eq:tailgauss}. 

(ii) The proof resembles very much the proof of Lemma \ref{lem:tail_bounds}\,(ii). We estimate
\begin{align*}
\P(\nu^*_{2^k}\ge t) &\le \binom{2^j}{2^k} \P(\nu_1\ge t)^{2^k} \le 
\Bigl(\frac{e\, 2^j}{2^k}\Bigr)^{2^k}\exp(-Nt^2/4\cdot 2^k)
=\Bigl[e\cdot 2^{j-k}\exp(-Nt^2/4)\Bigr]^{2^k}.
\end{align*}
\end{proof}

\subsection{Integrated absolute Wiener process}\label{sec:IW}

Let
\begin{equation}\label{eq:IntW01}
{\mathcal W}=\int_0^1 |W_s|\, ds
\end{equation}
be the integral of the absolute value of the Wiener process. The distribution of the random variable ${\mathcal W}$
is rather complicated, cf. \cite{K46,T93}. For our purpose, it will be sufficient to obtain concentration inequalities for ${\mathcal W}$ similar to those
for Gaussian and absolute values of Gaussian variables as given in Lemma \ref{lem:tail_bounds} and Lemma \ref{lem-prop-gv}. We use as a tool the integral of the square of the Wiener process
\begin{equation}\label{eq:IntW03}
{\mathcal S}=\left(\int_0^1 |W_s|^2\, ds\right)^{1/2} 
\end{equation}
in order to  get the tail bounds for \eqref{eq:IntW01}. 

\begin{lem}[{\bfseries \upshape Concentration inequalities for ${\mathcal W}$}]\label{lem:IW}\hfill 
\begin{enumerate}
\item[(i)] Let $t\ge 1$. Then
\begin{equation}\label{eq:IW01}
\P\Bigl({\mathcal W}\ge \frac{1}{2}+t\Bigr)
\le \exp(-t^2).
\end{equation}
\item[(ii)] There is a constant $c>0$ such that for every positive integer $N\ge 1$ and $(\mathcal{W}_j)_{j=1}^N$ i.i.d. as in \eqref{eq:IntW01}  it holds 
\begin{equation}\label{eq:IW02}
\P\Bigl(\frac{1}{N}\sum_{j=1}^N {\mathcal W}_j>t\Bigr)\le \exp(1-c\,Nt^2),\quad t>1.
\end{equation}
\item[(iii)] Let $0\le k<j$ be two integers and let $\nu=(\nu_0,\dots,\nu_{2^j-1})$ be a vector of independent random variables,
each equidistributed with $\frac{1}{N}\sum_{j=1}^N {\mathcal W}_j$. Then
\begin{equation}\label{eq:IntW04}
\P(\nu^*_{2^k}\ge t) 
\le \Bigl[e^2\cdot 2^{j-k}\exp(-c Nt^2)\Bigr]^{2^k},\quad t>1.
\end{equation}
\end{enumerate}
\end{lem}
\begin{proof} \emph{Step 1.}
We use the Karhunen-Lo\`eve expansion of the Wiener process (see \cite{Karhunen} or \cite[Chapter XI]{Loeve}), i.e., 
\begin{equation}\label{eq:IntW02}
W_t=\sqrt{2}\sum_{k=1}^\infty Z_k \frac{\sin((k-1/2)\pi t)}{(k-1/2)\pi},\quad t\in[0,1],
\end{equation}
where $(Z_k)_{k=1}^\infty$ is a sequence of independent standard Gaussian variables and the series converges in
$L_2$ uniformly over $t\in[0,1]$. We insert \eqref{eq:IntW02} into \eqref{eq:IntW03} and obtain
\[
{\mathcal W}^2\le {\mathcal S}^2=\int_0^1|W_s|^2ds=\sum_{k=1}^\infty \frac{Z_k^2}{(k-1/2)^2\pi^2}=\sum_{k=1}^\infty \alpha_k Z_k^2,
\]
where $\displaystyle \alpha_k=\frac{1}{(k-1/2)^2\pi^2}$. By \cite[Lemma 1]{LM00},
\begin{equation*}
\P\left(\sum_{k=1}^\infty \alpha_k(Z_k^2-1)\ge 2\|\alpha\|_2\sqrt{x}+2\|\alpha\|_\infty x\right)\le \exp(-x)
\end{equation*}
for every $x>0.$ Using that
\[
\|\alpha\|_1=\frac{1}{2},\quad \|\alpha\|_2=\frac{1}{\sqrt{6}},\quad \text{and}\quad  \|\alpha\|_\infty=\frac{4}{\pi^2}, 
\]
we obtain
\[
\P\Bigl({\mathcal W}\ge \frac{1}{2}+t\Bigr)\le \P\Bigl({\mathcal S}\ge\frac{1}{2}+t\Bigr)\le \exp\bigl(-t^2\bigr),\quad t\ge 1,
\]
which gives \eqref{eq:IW01}.\\
\emph{Step 2.} Using the properties of the Brownian motion, cf.  Definition \ref{def:BM}, we obtain
\[
\E {\mathcal W}=\int_0^1 \E|W_s|ds=\int_0^1 \E|W_s-W_0|ds=
\sqrt{\frac{2}{\pi}}\int_0^1\sqrt{s}ds=\sqrt{\frac{2}{\pi}}\cdot \frac{2}{3}\in\left(\frac12,1\right).
\]

Put $X={\mathcal W}-\E {\mathcal W}$. Then $\P(|X|>t)\le 1$ for $t\le \sqrt{2}$ and
\[
\P(|X|>t)=\P({\mathcal W}>t+\E {\mathcal W})\le \P\Bigl({\mathcal W}> t+\frac{1}{{2}}\Bigr)
\le \exp(-t^2)
\]

for $t>\sqrt{2}.$ We conclude that $\P(|X|>t)\le \exp(1-t^2/2)$ for all $t>0$, i.e., $X$ is a centered
subgaussian variable, cf. \cite[Definition 5.7]{Ver}.
Therefore, by the Hoeffding-type inequality, cf. \cite[Proposition 5.10]{Ver}, there is a constant $c_1>0$, such that
\[
\P\Bigl(\frac{1}{N}\sum_{j=1}^N {\mathcal W}_j>t+\E {\mathcal W}\Bigr)\le \exp(1-c_1\,Nt^2),\quad t>0,
\]
which, in turn, implies \eqref{eq:IW02} {for $t>1$ and $c>0$ small enough}.
\\

\emph{Step 3.} Finally, we conclude that if $\nu=(\nu_0,\dots,\nu_{2^j-1})$ is a vector of independent random variables, each equidistributed with
$\frac{1}{N}\sum_{j=1}^N {\mathcal W}_j$, then
\begin{equation*}
\P(\nu^*_{2^k}\ge t) 
\le \Bigl[e^2\cdot 2^{j-k}\exp(-c_2 Nt^2)\Bigr]^{2^k}
\end{equation*}
by an argument quite similar to the proof of Lemma \ref{lem-prop-gv} (ii).
\end{proof}

\subsection{Orlicz spaces}

Let us recall, that the Orlicz function $\Phi_2$ was defined in \eqref{eq:Orl01} and the corresponding Orlicz space $L_{\Phi_2}$ was introduced in \eqref{eq:Orl02}.
The following characterization is a special case of \cite[Theorem 10.3]{BenettRud80} adapted to {the domain $[0,1]^d$.}
We include its short proof to make our presentation self-contained.

\begin{thm}\label{thm:Orlicz}
A measurable function $f$ on $[0,1]^d$ belongs to $L_{\Phi_2}([0,1]^d)$ if, and only if, there exists $c>0$ such that
\begin{equation}\label{eq:Orlicz1}
f^*(t)\le c\sqrt{\log(1/t)+1},\quad 0<t<1.
\end{equation}
Moreover, the expression
\begin{align*}
    \|f\|_{\Phi_2^*}:=\sup_{0<t<1}\frac{f^*(t)}{\sqrt{\log(1/t)+1}}
\end{align*}
is equivalent to $\|f\|_{\Phi_2}$. 
\end{thm}
\begin{proof}
If $f$ satisfies \eqref{eq:Orlicz1} for some $c$, then we obtain for $\lambda=2c$
\begin{align*}
\int_{[0,1]^d}\Bigl[\exp(f(x)^2/\lambda^2)-1\Bigr]dx&=\int_0^1\Bigl[\exp(f^*(t)^2/\lambda^2)-1\Bigr]dt \le
\int_0^1\Bigl[\exp\Bigl(\frac{c^2(\log(1/t)+1)}{\lambda^2}\Bigr)-1\Bigr]dt\\
&=\int_0^1 \Bigl[\exp\Bigl(\frac{\log(1/t)+1}{4}\Bigr)-1 \Bigr] dt<1.
\end{align*}
Hence, $f\in L_{\Phi_2}([0,1]^d)$ and $\|f\|_{\Phi_2}\le 2 \|f\|_{\Phi_2^*}$.

If, on the other hand, $f\in L_{\Phi_2}([0,1]^d)$ with $\|f\|_{{\Phi_2}}\le 1$, then
\begin{align*}
1&\ge \int_{[0,1]^d} \Phi_2(|f(x)|)dx=
\int_0^1\Phi_2(f^*(s))ds\ge \int_0^t \Bigl[\exp(f^*(s)^2)-1\Bigr]ds\\
&\ge t\Bigl[\exp(f^*(t)^2)-1\Bigr],
\end{align*}
i.e., $f^*(t)^2\le \log(1+1/t)\le \log(e/t)=1+\log(1/t)$ for every $0<t<1$ and \eqref{eq:Orlicz1} follows.
\end{proof}

We need also another characterization of the norm of
$L_{\Phi_2}([0,1]^d)$. Its proof can be found in \cite[Theorem  3.4]{Ciesiel93} but again we include it for the reader's convenience.

\begin{thm}\label{thm:phi2}
A measurable function $f$ on $[0,1]^d$ belongs to $L_{\Phi_2}([0,1]^d)$ if, and only if, there exists $c>0$ such that
\begin{align*}
    \|f\|_p\leq c \sqrt p \quad\text{holds for all}\quad  p\geq 1.
\end{align*}
Moreover,
\begin{align}\label{eq:wurzelnorm}
    \|f\|_{(\Phi_2)}:=\sup_{p\geq1}\frac{\|f\|_p}{\sqrt{p}} 
\end{align}
is an equivalent norm on $L_ {\Phi_2}([0,1]^d)$.
\end{thm}
\begin{proof}
First of all we show $\|f\|_{(\Phi_2)}\leq\|f\|_{\Phi_2}$. To that end, let $f\in L_{\Phi_2}([0,1]^d)$ be given with $\|f\|_{\Phi_2}\leq1$. Then by the power series of the
exponential function we estimate for any $n\in\N$ 
\begin{align*}
    1\geq\int_{[0,1]^d}\Phi_2(|f(x)|)dx=\int_{[0,1]^d}\sum_{k=1}^\infty\frac{|f(x)|^{2k}}{k!}dx
    \geq\frac{1}{n!}\|f\|_{2n}^{2n}.
\end{align*}
Now using $n!\leq n^n$ we obtain
\begin{align*}
    \frac{\|f\|_{2n}^{2n}}{n^n}\leq1\quad
    \text{which is equivalent to}\quad
    \frac{\|f\|_{2n}}{\sqrt{2n}}\leq\frac{1}{\sqrt2}.
\end{align*}
If $1\le p<2$, we obtain
\[
\frac{\|f\|_p}{\sqrt{p}}\le \|f\|_2\le 1.
\]
If $2<p<\infty$, we choose the unique $n\in\N$ with $n\ge 2$ such that $2(n-1)< p\leq2n$ and obtain
\begin{align*}
    \frac{\|f\|_p}{\sqrt{p}}\leq\frac{\|f\|_{2n}}{\sqrt{2(n-1)}}\le \frac{\sqrt{n}}{\sqrt{2(n-1)}}\leq1,
\end{align*}
which finishes the first step.

In the second step, we are going to show $\|f\|_{\Phi_2}\leq C\|f\|_{(\Phi_2)}$. To that end, we choose $f$ such that \eqref{eq:wurzelnorm} is finite.
Using Stirling's formula we can fix an $\lambda_0>1$ such that $\lambda_0n!\geq(n/e)^n$  holds for all $n\in\N$. We estimate with
the help of the power series of the exponential function
\begin{align*}
    \int_{[0,1]^d}\Phi_2\left(\frac{|f(x)|}{\lambda}\right)dx&=\int_{[0,1]^d}\sum_{n=1}^\infty\frac{|f(x)|^{2n}}{\lambda^{2n}n!}dx\leq\lambda_0\sum_{n=1}^\infty\frac{(2\e)^n}{\lambda^{2n}}\frac{\|f\|_{2n}^{2n}}{(2n)^n}\\
    \intertext{and now choosing $\lambda=\sqrt{2e(1+\lambda_0)}\|f\|_{(\Phi_2)}$ gives}
 \int_{[0,1]^d}\Phi_2\left(\frac{|f(x)|}{\lambda}\right)dx  &\leq\lambda_0\sum_{n=1}^\infty(1+\lambda_0)^{-n}=1.
\end{align*}
This shows $\|f\|_{\Phi_2}\leq\lambda=\sqrt{2e(1+\lambda_0)}\|f\|_{(\Phi_2)}$ and finishes the proof.
\end{proof}

In the refined analysis concerning the regularity of Brownian paths, we need also a generalization of Theorem \ref{thm:Orlicz}.
First, we define a scale of Orlicz functions $\Phi_{2,A}$, where $0<A\le 1$ is a real parameter, via 
\begin{equation}\label{eq:Orlicz_new1}
\Phi_{2,A}(u)=\begin{cases}
u^2,\quad &0<u\le 1,\\
\displaystyle\exp\Bigl(\frac{u^2-1}{A}\Bigr),\quad &1< u<\infty.
\end{cases}
\end{equation}
It is easy to see that this scale of Orlicz functions fulfills the following estimates for all $u>0$ and all $0<A\leq1$ 
\begin{align*}
    \Phi_2\left(\frac{u}{\sqrt{2}}\right)
    \le \Phi_{2,1}(u)\le \Phi_{2,A}(u)\le\Phi_2\left(\frac{u}{\sqrt{A}}\right).
\end{align*}

Therefore the Orlicz space associated to $\Phi_{2,A}$ coincides with $L_{\Phi_2}$
for every $0<A\le 1$. Nevertheless, the equivalence constants in the  respective norms will depend on $A$. It is quite interesting (and of a crucial importance for us) that the following simple expression 
\begin{equation}\label{eq:OrliczA}
\|f\|_{2,A}^{(1)}:=\sup_{0<t<1}\frac{f^*(t)}{\sqrt{A\log(1/t)+1}}
\end{equation}
is equivalent to the Orlicz norm associated with $\Phi_{2,A}$ (which we denote by $\|f\|_{2,A}$) and that the equivalence constants are independent on the parameter $A\in(0,1]$.

\begin{thm}\label{thm:OrliczA}
Let $0<A\le 1$ and let $f$ be a measurable function on $[0,1]^d$. Then
\[
\|f\|_{2,A}^{(1)}\le \|f\|_{2,A}\le 4\,\|f\|_{2,A}^{(1)}.
\]
\end{thm}
\begin{proof}
Let $\|f\|_{2,A}\le 1$ and let $0<t<1$. If $f^*(t)\le 1$, then also $f^*(t)\le \sqrt{A\log(1/t)+1}$ and there is nothing to prove. If $f^*(t)>1$, then we estimate
\begin{align*}
1\ge \int_{[0,1]^d}\Phi_{2,A}(|f(s)|)ds\ge \int_0^t\Phi_{2,A}(f^*(s))ds\ge t\,\Phi_{2,A}(f^*(t))=t\exp\Bigl(\frac{f^*(t)^2-1}{A}\Bigr).
\end{align*}
By simple algebraic manipulations, it follows that
\[
f^*(t)\le \sqrt{A\log(1/t)+1},\quad 0<t<1.
\]

Let, on the other hand, $\|f\|_{2,A}^{(1)}=c$. Then, putting $\lambda=4c$ we obtain $\frac{f^*(t)^2}{\lambda^2}\leq \frac{A\log(1/t)+1}{16} $ and estimate
\begin{align*}
\int_{[0,1]^d}\Phi_{2,A}\biggl(\frac{|f(x)|}{\lambda}\biggr)dx&=\int_0^1\Phi_{2,A}\biggl(\frac{f^*(t)}{\lambda}\biggr)dt                
\le \int_0^1 \exp\Bigl(\frac{f^*(t)^2/\lambda^2-1}{A}\Bigr)dt+\int_0^1 \frac{f^*(t)^2}{\lambda^2}dt\\
&\le \int_0^1 \exp\Bigl(\frac{\log(1/t)}{16}-\frac{15}{16A}\Bigr)dt 
+ \int_0^1 \frac{A\log(1/t)+1}{16}dt\\
&\le \exp\Bigl(-\frac{15}{16}\Bigr)\int_0^1\exp\Bigl(\frac{\log(1/t)}{16}\Bigr)dt+ \frac{1}{16}\int_0^1\Bigl(\log(1/t)+1\Bigr)dt\\
&=\exp\Bigl(-\frac{15}{16}\Bigr)\cdot\frac{16}{15}+\frac{1}{8}\le 1,
\end{align*}
which implies that $\|f\|_{2,A}\le \lambda=4c.$
\end{proof}

{\bf Acknowledgement:} We would like thank Mikhail Lifshits, Werner Linde, and Winfried Sickel for very helpful discussions.


\thebibliography{99}
\bibitem{Babenko} K.I. Babenko,  {On the approximation of a certain class of periodic functions of several
variables by trigonometric polynomials}, {\em Dokl. Akad. Nauk USSR}, {\bf 132}, 982--985 (1960). 
{\em English transl. in Soviet Math. Dokl.}, 1 (1960).
\bibitem{BenettRud80} C. Bennett and K. Rudnick, {On Lorentz-Zygmund spaces},
 {\em Dissertationes Math.} (Rozprawy Mat.)  {\bf 175}, 67 pp, (1980).
\bibitem{BO11} \'A. B\'enyi and T. Oh,  {Modulation spaces, Wiener amalgam spaces, and Brownian motions}, {\em Adv. Math.} {\bf 228} (5),  2943--2981 (2011).
\bibitem{BL1} D. Bilyk and M. T. Lacey,  {On the small ball inequality in three dimensions}, {\em Duke Math. J.} {\bf 143}  (1), 81--115 (2008).
\bibitem{BLV1} D. Bilyk, M. T. Lacey, and A. Vagharshakyan,  {On the small ball inequality in all dimensions}, {\em J. Funct. Anal.} {\bf 254} (9), 2470--2502 (2008).
\bibitem{BG04} H.-J. Bungartz and M. Griebel,  {Sparse grids}, {\em Acta Numerica} {\bf 13}, 147--269 (2004).  
\bibitem{Byrenheid} G. Byrenheid, {\em Sparse representation of multivariate functions based on discrete point evaluations},
PhD Thesis, Bonn (2018).
\bibitem{Ciesiel91} Z. Ciesielski, {Modulus of smoothness of the Brownian paths in the $L_p$ norm},
{\em Constructive theory of functions} (Varna, Bulgaria, 1991), pages 71--75 (1992).
\bibitem{Ciesiel93} Z. Ciesielski, {Orlicz spaces, spline systems, and Brownian motion}, {\em Constr. Approx.}  {\bf 9}, 191--208 (1993).
\bibitem{CKR93} Z. Ciesielski, G. Kerkyacharian, and B. Roynette,  {Quelques espaces fonctionnels associ\'es \`a des processus gaussiens}, {\em Studia Math.} {\bf 107} (2), 171--204 (1993).
\bibitem{CR81} M. Cs\"org\H{o} and P. R\'ev\'esz, Strong approximations in probability and statistics. Probability
and Mathematical Statistics,  {\em Academic Press}, Inc. [Harcourt Brace Jovanovich, Publishers], New York-London (1981).
\bibitem{dal04} R.C. Dalang,  {Level Sets and Excursions of the Brownian Sheet},  Topics in Spatial Stochastic Processes, {\em Lecture Notes in Mathematics} {\bf 1802}, Springer, Berlin, Heidelberg (2003). 
\bibitem{DomTikh} O. Dom\'\i nguez and S. Tikhonov, {Function spaces of logarithmic smoothness: embeddings and characterizations},
{\em Mem. Amer. Math. Soc.} (to appear). 
\bibitem{Dunker} T. Dunker,  {Estimates for the small ball probabilities of the fractional Brownian sheet}, {\em J. Theoret. Probab.} {\bf 13}  (2), 357--382 (2000). 
\bibitem{DTU18} D. Dung, V. Temlyakov, and T. Ullrich, {\em Hyperbolic cross approximation}, Springer (2018). 
\bibitem{Fab09} G. Faber, {{\"Uber stetige Funktionen}}, {\em Math. Ann.}  {\bf 66}, 81--94, (1909). 
\bibitem{FaLeo} W. Farkas and. H.-G. Leopold, \emph{Characterisations of function spaces of generalised smoothness}, {\em Ann. Mat. Pura Appl.} {\bf 185} (1), 1--62 (2006). 
\bibitem{Haar10} A. Haar,  {Zur Theorie der orthogonalen Funktionensysteme}, {\em Math. Ann.} {\bf 69}, 331--371 (1910). 
\bibitem{HansenDiss} M. Hansen, {\em Nonlinear approximation and function spaces of dominating mixed smoothness}, PhD Thesis,  Jena (2010). 
\bibitem{Herren} V. Herren, {Lévy-type processes and Besov spaces}, {\em Potential Analysis}, {\bf 7} (3), 689--704 (1997). 
\bibitem{Hummel} F. Hummel,  {Sample paths of white noise in spaces with dominating mixed smoothness},  {\em Banach J. Math. Anal.} {\bf  15} (3), Paper No. 54, 38 pp. (2021)
\bibitem{HV08} T.P. Hyt\"onen and M.C. Veraar, {On Besov regularity of Brownian motions in infinite dimensions}, {\em Probab. Math. Statist.}  {\bf 28}, 143--162 (2008). 
\bibitem{K46} M. Kac,  {On the average of a certain Wiener functional and a related limit theorem in calculus of probability}, {\em Trans. Amer. Math. Soc.} {\bf 59}, 401--414 (1946). 
\bibitem{Kahane} J.P. Kahane,  {\em Some random series of functions}, 2nd ed. Cambridge
 University Press, London (1985).
\bibitem{KaLi} G. A. Kalyabin and P. I. Lizorkin, {Spaces of functions of generalized smoothness}, {\em Math. Nachr.} {\bf 133}, 7--32 (1987).
\bibitem{Kamont2} A. Kamont,  {Isomorphism of some anisotropic Besov and sequence spaces}, {\em Studia Math.} {\bf 110} (2), 169--189 (1994).
\bibitem{Kamont1} A. Kamont,  {On the fractional anisotropic Wiener field}, {\em Probab. Math. Statist.} {\bf 16} (1), 85--98, (1996).  
\bibitem{Karhunen} K. Karhunen,  {\"Uber lineare Methoden in der Wahrscheinlichkeitsrechnung},
{\em Ann. Acad. Sci. Fennicae Ser. A. I. Math.-Phys.} {\bf 37}, 79pp. (1947). 
\bibitem{Khos} D. Khoshnevisan, {\em Multiparameter processes. An introduction to random fields}, Springer Monographs in Mathematics,  Springer-Verlag, New York (2002). 
\bibitem{KhosXiao} D. Khoshnevisan and Y. Xiao,  {Lévy processes: capacity and Hausdorff dimension}, {\em Ann. Probab.} {\bf 33} (3), 841--878 (2005). 
\bibitem{KL02} T. K\"uhn and W. Linde,  {Optimal series representation of fractional Brownian sheets}, {\em Bernoulli} {\bf 8}, 669--696 (2002). 
\bibitem{KL93} J. Kuelbs and W.V. Li,  {Metric entropy and the small ball problem for Gaussian measures}, {\em J. Funct. Anal.} {\bf 116}, 133--157 (1993). 
\bibitem{KL93'} J. Kuelbs and W. Li, {Small ball estimates for Brownian motion and the Brownian sheet}, {\em J. Theoret. Probab.} {\bf 6} (3), 547--577 (1993). 
\bibitem{LM00} B. Laurent and P. Massart,  {Adaptive estimation of a quadratic functional by model selection}, {\em Ann. Statist.} {\bf 28}  (5), 1302--1338 (2000). 
\bibitem{Levy} P. L\'evy, {\em Th\'eorie de l’addition des variables al\'eatoires}. Monographies des Probabilit\'es; calcul des probabilit\'es et ses applications {\bf  1},  Paris (1937).
\bibitem{Levy2} P. L\'evy, {\em Processus stochastiques et mouvement brownien}, Gauthier-Villars \& Cie, Paris (1965).
\bibitem{Linde99} W. Li and W. Linde,  {Approximation, metric entropy and small ball estimates for Gaussian measures}, {\em Ann. Probab.} {\bf 27} (3), 1556--1578 (1999). 
\bibitem{Loeve} M. Lo\`{e}ve, {\em Probability theory II.}, Graduate Texts in Mathematics {\bf 46}, Springer-Verlag, New York-Heidelberg  (1978). 
\bibitem{MP} P. M\"orters and Y. Peres, {\em Brownian motion}, Cambridge Series in Statistical and Probabilistic Mathematics {\bf 30},  Cambridge University Press, Cambridge (2010). 
\bibitem{Moura} S. D. Moura,  {Function spaces of generalized smoothness}, {\em Dissertationes Math.} {\bf 398}, 88 pp. (2001). 
\bibitem{OSK18} M. Ondrej\'at, P. \v{S}imon, and M. Kupsa,  {Support of solutions of stochastic differential equations in exponential Besov-Orlicz spaces}, {\em Stoch. Anal. Appl.} {\bf 36} (6), 1037--1052 (2018). 
\bibitem{OV20} M. Ondrej\'at and M. Veraar,  {On temporal regularity of stochastic convolutions in 2-smooth Banach spaces}, {\em Ann. Inst. Henri Poincar\'e Probab. Stat.} {\bf 56} (3), 1792--1808 (2020). 
\bibitem{Peetre} J. Peetre, New thoughts on Besov spaces, {\em Duke University Mathematics Series} {\bf 1},  Duke University, Durham, N.C.  (1976). 
\bibitem{PiSi} L. Pick and W. Sickel,  {Several types of intermediate Besov-Orlicz spaces}, {\em Math. Nachr.} {\bf 164}, 141--165 (1993). 
\bibitem{R93} B. Roynette,  {Mouvement brownien et espaces de Besov}, {\em Stochastics and Stochastics Rep.} {\bf 43}, 221--260 (1993). 
 \bibitem{SiRu} T. Runst and W. Sickel, {\em Sobolev spaces of fractional order, Nemytskij operators, and nonlinear partial differential equations}, De Gruyter Series in Nonlinear Analysis and Applications {\bf 3},  Walter de Gruyter \& Co., Berlin (1996). 
\bibitem{SchTr87} H.-J. Schmeisser and H. Triebel. {\em Topics in Fourier analysis and function spaces},  John Wiley \& Sons Ltd., Chichester (1987). 
\bibitem{Schilling1} R. L. Schilling,  {On Feller processes with sample paths in Besov spaces}, {\em Math. Ann.} {\bf 309} (4), 663--675 (1997). 
\bibitem{Schilling3} R. L. Schilling,  {Function spaces as path spaces of Feller processes}, {\em Math. Nachr.} {\bf 217}, 147--174 (2000). 
\bibitem{SU09} W. Sickel and T. Ullrich,  {Tensor products of Sobolev-Besov spaces and applications to approximation from the hyperbolic cross}, {\em J. Approx. Theor.} {\bf 161}, 748--786  (2009). 
\bibitem{T93} L. Tak\'acs,  {On the distribution of the integral of the absolute value of the Brownian motion}, {\em Ann. Appl. Probab.} {\bf 3}  (1), 186--197  (1993).  
\bibitem{Tal94} M. Talagrand,  {The small ball problem for the Brownian sheet}, {\em Ann. Probab.} {\bf 22}, 1331--1354 (1994). 
\bibitem{Teml93} V.N. Temlyakov, {\em Approximation of periodic functions}, Nova Science Publishes, Inc., New York. (1993). 
\bibitem{Tong} Y. L. Tong, {\em The multivariate normal distribution}, Springer Series in Statistics, Springer-Verlag, New York (1990). 
\bibitem{TrFS1} H. Triebel, {\em Theory of function spaces}, Monographs in Mathematics {\bf 78}, Birkh\"auser Verlag, Basel (1983). 
\bibitem{Triebel10} H. Triebel, {\em Bases in function spaces, sampling, discrepancy, numerical integration}, EMS Tracts in Mathematics, {\bf 11}, European Mathematical Society, Z\"urich (2010).
\bibitem{Triebel12} H. Triebel, \emph{Faber systems and their use in sampling, discrepancy, numerical integration}, European Mathematical Society, Z\"urich (2012).
\bibitem{Triebel19} H. Triebel, \emph{Function spaces with dominating mixed smoothness}, European Mathematical Society, Z\"urich (2019). 
\bibitem{TudorXiao} C. A. Tudor and Y. Xiao,  {Sample path properties of bifractional Brownian motion}, {\em Bernoulli}  {\bf 13} (4), 1023--1052 (2007). 
\bibitem{Veraar} M. C. Veraar,  {Regularity of Gaussian white noise on the $d$-dimensional torus}, In: M. Nawrocki, W. Wnuk (eds.) {\em  Marcinkiewicz Centenary Volume}  {\bf 95},  {  Banach Center Publication}, pp. 385--398. Polish Acad. Sci. Inst. Math., Warsaw (2011). 
\bibitem{Ver} R. Vershynin, {Introduction to the non-asymptotic analysis of random matrices}, In: {\em Compressed sensing,  Cambridge Univ. Press}, Cambridge,  210--268 (2012).
\bibitem{Vyb06} J. Vyb\'\i ral, {Function spaces with dominating mixed smoothness}, {\em Dissertationes Math.} {\bf 436}, 73 pp. (2006). 
\bibitem{Walsh} J.B. Walsh, {An Introduction to Stochastic Partial Differential Equations}, École d’été de probabilités de Saint-Flour XIV-1984, {\em Lecture Notes in Mathematics} {\bf 1180}, Springer-Verlag, New York,  265--439 (1986). 
\bibitem{Wang07} W. Wang,  {Almost-sure path properties of fractional brownian sheet}, {\em Annales de l'I.H.P. Probabilités et statistiques, Tome } {\bf 43} (5), 619--631 (2007). 
\bibitem{Xiao1} Y. Xiao,  {Sample path properties of anisotropic Gaussian random fields, A minicourse on stochastic partial differential equations}, 145--212,
{\em Lecture Notes in Math.} {\bf 1962}, Springer, Berlin (2009).
\end{document}